\documentclass[dvips,aop,preprint]{imsart}

\RequirePackage[OT1]{fontenc} \RequirePackage{amsthm,amsmath,
amsfonts,amsthm,amscd,amssymb,graphicx}
\RequirePackage[numbers]{natbib}
\RequirePackage[colorlinks,citecolor=blue,urlcolor=blue]{hyperref}
\RequirePackage{hypernat}

\startlocaldefs
\numberwithin{equation}{section}
\theoremstyle{plain}

\newtheorem{theorem}{Theorem}[section]

\newtheorem{lemma}[theorem]{Lemma}

\newtheorem{corollary}[theorem]{Corollary}

\newtheorem{definition}[theorem]{Definition}

\endlocaldefs

\begin{document}

\begin{frontmatter}
\title{A Khintchine Decomposition for Free Probability.}
\runtitle{Khintchine Decomposition.}

\begin{aug}
\author{\fnms{John D.} \snm{Williams}\ead[label=e1]{jw32@indiana.edu}}

\runauthor{John D. Williams}

\address{Indiana University\\
Rawles Hall\\
Bloomington, In \ \ 47405\\
\printead{e1}\\
}
\end{aug}

\begin{abstract}
Let $\mu$ be a probability measure on the real line.  In this paper
we prove that there exists a decomposition $\mu = \mu_{0} \boxplus
\mu_{1} \boxplus \cdots \boxplus \mu_{n} \boxplus \cdots$ such that
$\mu_{0}$ is infinitely divisible and $\mu_{i}$ is indecomposable
for $i \geq 1$. Additionally, we prove that the family of all
$\boxplus$-divisors of a measure $\mu$ is compact up to translation.
Analogous results are also proven in the case of multiplicative
convolution.
\end{abstract}



\end{frontmatter}

\section{Introduction}

In classical probability theory, it has long been known that the set
of all convolution divisors of a random variable is compact up to
translation. That is, given a family of decompositions $\mu =
\mu_{1,i} \ast \mu_{2,i}$ with $i \in I$, the families $\{ \mu_{j,i}
\}_{i \in I , j = 1,2}$ can be translated to form sequentially
compact families $\{ \hat{\mu}_{i,j} \}_{i \in I , j = 1,2}$ so that
$\mu = \hat{\mu}_{1,i} \ast \hat{\mu}_{2,i}$ for all $i \in I$. The
proof of this result is a simple application of L\'evy's Lemma (see
Chapter $5$ in \cite{Ln} for a full account of the classical case).
 This compactness lemma serves as the cornerstone for the proof of
the following classical result of Khintchine.
\begin{theorem}\label{classic}
Let $\mu$ be a probability measure.  Then there exist measures
$\mu_{i}$ with $i = 0,1,2, \ldots $ such that $\mu_{0}$ is
$\ast$-infinitely divisible, $\mu_{i}$ is indecomposable for $i = 1,
2, \ldots $, and $\mu = \mu_{0} \ast \mu_{1} \ast \mu_{2} \ast
\cdots$. This decomposition is not unique.
\end{theorem}
The equation $\mu = \mu_{0} \ast \mu_{1} \ast \mu_{2} \ast \cdots$
is in the sense that in the weak$^{\ast}$ topology we have that
$\lim_{n \uparrow \infty} \mu_{0} \ast \mu_{1} \ast \cdots \ast
\mu_{n} = \mu$. This type of equality will be used throughout the
paper without further comment.

In free probability theory, the corresponding compactness and
decomposition theorems have hitherto been absent from the
literature.  Partial results of the corresponding compactness
theorem are near trivialities.  Indeed, consider a $W^{\ast}$
probability space $(A, \tau)$ and a random variable $X \in A$ with
mean $0$ and finite variance. Let $X = X_{1} + X_{2}$ be a
decomposition with the $X_{i}$'s freely independent and of mean $0$.
Then the equation $\tau(X^{2}) = \tau(X_{1}^{2}) + \tau(X_{2}^{2})$
would imply the necessary tightness result when applied to families
of decompositions.

It is the first aim of this paper to prove the corresponding
tightness results in the fullest possible generality.  That is, we
make no assumptions as to the finiteness of moments.  It is the
second aim of this paper to prove versions of Theorem \ref{classic}
for additive and multiplicative free convolution.

This paper is organized as follows:   in Section $2$ we give the
background and terminology of additive free convolution; in Section
$3$ we state and prove a number of compactness results for families
of decompositions with respect to additive free convolution; in
Section $4$ we prove the existence of the Khintchine decomposition
with respect to additive free convolution; Sections $5$, $6$ and $7$
are the respective analogues of Sections $2$, $3$ and $4$ but with
regard to multiplicative free convolution for measures supported on
the positive real numbers; in Section $8$ we give the background and
terminology for multiplicative free convolution of measures
supported on the unit circle; in Section $9$ we prove the existence
of the Khintchine decomposition for measures supported on the unit
circle; in Section $10$ we provide applications of our compactness
results.

\section{Background and Terminology for Additive Free Convolution}
We refer to \cite{VDN} for a full account of the basics of free
probability theory.

Let $(A,\tau)$ be a $W^{\ast}$ probability space.  We say that a
family of unital subalgebras $\{ A_{i} \}_{i \in I}$ are
\textit{freely independent} if $\tau(x_{i_{1}} x_{i_{2}} \cdots
x_{i_{n}}) = 0$ for  $x_{i_{j}} \in A_{i_{j}}$ whenever $i_{j} \neq
i_{j + 1}$ for $j= 1, \ldots , n-1$ and $\tau(x_{i_{k}}) = 0$ for $k
= 1, \ldots , n$.  We say that random variables $x,y \in A$ are
freely independent if the unital algebras that they generate in $A$
satisfy the above definition.

Assume that $A \subset B(H)$.  We say that a not necessarily bounded
operator $x$ is affilitated with $A$ (in symbols, $x \eta A$) if the
spectral projections of $x$ are elements in $A$.  Equivalently, $x
\eta A$ if for every $y \in A'$ (the commutant of A), we have that
$yx \subset xy$.  This expanded class of random variables allows us
to study measures with unbounded support.

Let $x \ \eta \ A$ be a self-adjoint random variable with
distribution $\mu$, a probability measure supported on $\mathbb{R}$.
We associate to $\mu$ its  Cauchy transform:

$$G_{\mu}(z) = \int_{\mathbb{R}} \frac{d\mu(t)}{z-t} = \tau ((z-x)^{-1})$$
Observe that $zG_{\mu}(z) \rightarrow 1$ as $z \rightarrow \infty$
nontangentially.  It follows that $G_{\mu}$ is univalent on a set of
the form $\Gamma_{\alpha, \beta} = \{z \in \mathbb{C^{+}} : \Im{(z)}
> \alpha , \ \Im{(z)} > \beta \Re{(z)} \}$ for sufficiently large $\alpha , \beta > 0$.
  Throughout this paper we shall refer to a set of this type as a
\textit{Stolz angle}. The set $G_{\mu}(\Gamma_{\alpha, \beta})$
contains a set of the form $\Lambda_{\alpha', \beta'} = \{ z\in
\mathbb{C}^{-} : 0 < \Im{(z)} \leq \alpha' , \ \beta' \Re{(z)} <
\Im{(z)} \}$ on which we have a well defined left inverse,
$G_{\mu}^{-1}$.   The function $R_{\mu} (z) = G_{\mu}^{-1} (z) -
1/z$ is called the R-transform of $\mu$. First proved in \cite{V1},
the following equality is fundamental in free probability theory:
$$R_{\mu \boxplus \nu} (z) = R_{\mu}(z) + R_{\nu}(z)$$

In what follows, it will be more convenient to consider the
following functions:
$$F_{\mu}(z) = \frac{1}{G_{\mu}(z)} $$
$$\varphi_{\mu}(z) = F_{\mu}^{-1}(z) - z = R_{\mu} (1 / z)$$
These functions are refered to as the $F$ and Voiculescu transform,
respectively.  They have the following properties which are proven
to various degrees of generality in \cite{BV1}, \cite{V1}, and
\cite{Ma}:
\begin{enumerate}
\item
$|F_{\mu}(z) - z| = o(|z|)$ uniformly as $|z| \rightarrow \infty$ in
$ \Gamma_{\alpha, \beta}$ for all $\alpha , \beta > 0$.

\item
$\Im{(F_{\mu(z)})} \geq \Im{(z)}$ for all $z \in \mathbb{C}^{+}$.

\item
$F_{\mu}$ has a well defined left inverse on $\Gamma_{\alpha,
\beta}$ for some $\alpha , \beta > 0$ (hence, the Voiculescu
transform is defined on this set).

\item
There exist $\alpha , \beta > 0$ such that $\varphi_{\mu \boxplus
\nu} (z) = \varphi_{\mu}(z) + \varphi_{\nu} (z)$ when $z \in
\Gamma_{\alpha , \beta}$.

\item
$F_{\mu \boxplus \delta_{c}} (z) = F_{\mu} (z-c)$ and $\varphi_{\mu
 \boxplus \delta_{c}} = c + \varphi_{\mu}(z)$ for $c \in \mathbb{R}$.

\end{enumerate}

Given a decomposition $\mu = \mu_{1} \boxplus \mu_{2}$, it was shown
in \cite{V3} and \cite{Bi} that there exist analytic subordination
functions $\omega_{i} : \mathbb{C}^{+} \rightarrow \mathbb{C}^{+}$
such that:

\begin{enumerate}
\item
$F_{\mu}(z) = F_{\mu_{i}}(\omega_{i}(z))$ for $z \in \mathbb{C}^{+}$
and $i = 1,2$.

\item
$\lim_{y \uparrow + \infty} \frac{\omega_{i} (iy)}{iy} = 1$ for $i =
1,2$.

\item
$\omega_{1} (z) + \omega_{2} (z) = z + F_{\mu} (z)$
\end{enumerate}

Observe that $\omega_{i}$ and $F_{\mu}$ satisfy the same asymptotic
properties in $(2)$ above.  A classical result, due to Nevanlinna
(whose full account can be found in \cite{AG}, Volume $2$, page
$7$),  implies that these functions have the following
representation:

$$\omega_{i} (z)  = r_{i} + z + \int_{- \infty}^{\infty} \frac{1 + tz}{ z - t} d\sigma_{i}(t)$$
$$F_{\mu}(z) =  r + z + \int_{-\infty}^{\infty} \frac{1 + tz}{ z - t} d\sigma(t)$$
where $r , r_{i} \in \mathbb{R}$ and $\sigma$, $\sigma_{i}$ are
positive, finite measures which are uniquely determined by
$\omega_{i}$ and $F_{\mu}$. Observe that property $(3)$ above and
uniqueness imply that $r_{1} + r_{2} = r$ and $\sigma_{1} +
\sigma_{2} = \sigma$.

We denote by $\mathfrak{F}_{\mu} (t) = \mu((-\infty, t])$ the
cumulative distribution function of $\mu$.  This function is used to
define two metrics on the space of probability measures, namely the
Kolmogorov and L\'evy metric, $d_{\infty}$ and $d$ respectively.
These are defined as follows:
$$d_{\infty} (\mu , \nu) = \sup_{t \in \mathbb{R}} |\mathfrak{F}_{\mu}(t) - \mathfrak{F}_{\nu}(t)| $$
$$d(\mu, \nu) = \inf \{\epsilon > 0 : \mathfrak{F}_{\mu}(t-\epsilon) - \epsilon \leq \mathfrak{F}_{\nu}(t) \leq \mathfrak{F}_{\mu}(t + \epsilon) + \epsilon \} $$
The L\'evy metric induces the weak topology on the space of
probability measures on the line while the Kolmogorov metric induces
a stronger topology which we call the Kolmogorov topology.  We have
the the following facts, first proven in \cite{BV1}, which will be
used throughout, often without reference:

\begin{lemma}
Let $\mu_{n}$ and $\nu_{n}$ converge to probability measures $\mu$
and $\nu$ respectively in the weak$^{\ast}$ (resp., Kolmogorov)
topology. Then $\mu_{n} \boxplus \nu_{n}$ converges to $\mu \boxplus
\nu$ in the weak$^{\ast}$ (resp., Kolmogorov) topology.
\end{lemma}

The proof of this lemma relies on the following inequalities which
will be used in what follows:

$$d(\mu \boxplus \nu ,  \mu' \boxplus \nu') \leq d(\mu , \mu') + d(\nu , \nu') $$

$$d_{\infty}(\mu \boxplus \nu ,  \mu' \boxplus \nu') \leq d_{\infty}(\mu , \mu') + d_{\infty}(\nu , \nu') $$

The next two lemmas were first proven in Section $5$ of \cite{BV1}.

\begin{lemma}\label{tightdomain}
Let $\{ \mu_{n} \}_{n \in \mathbb{N}}$ be a tight sequence of
measures.  Then there exists a Stolz angle $\Gamma_{\alpha, \beta}$
such that the functions $|F_{\mu_{n}}(z) - z| = o(z)$ uniformly as
$|z| \rightarrow \infty$ in this set.  In particular, the functions
$F^{-1}_{\mu_{n}}$ exist on a common domain for all $n$.
\end{lemma}

\begin{lemma}\label{BVconv}
Let $\{ \mu_{n} \}_{n \in \mathbb{N}}$ be a sequence of probability
measures on $\mathbb{R}$.  The following assertions are equivalent:

\begin{enumerate}
\item
The sequence $\{ \mu_{n} \}_{n \in \mathbb{N}}$ converges in the
weak$^{\ast}$ topology to a probability measure $\mu$.

\item
There exist $\alpha , \beta > 0$ such that the functions $\{
\varphi_{\mu_{n}} \}_{n \in \mathbb{N}}$ are defined and converge
uniformly on compact subsets of $\Gamma_{\alpha , \beta}$ to a
function $\varphi$ and $\varphi_{\mu_{n}} (z) = o(z)$ uniformly in
$n$ as $|z| \rightarrow \infty$, $z \in \Gamma_{\alpha, \beta}$.

\end{enumerate}
Moreover, if $(1)$ and $(2)$ are satisfied we have that $\varphi =
\varphi_{\mu}$ in $\Gamma_{\alpha , \beta}$.

\end{lemma}

\begin{definition}
A probability measure $\mu$ on the real line is said to be
$\boxplus$-infinitely divisible if for every $n \in \mathbb{N}$
there exists a measure $\mu_{1/n}$ such that $\mu = \mu_{1/n}
\boxplus \cdots \boxplus \mu_{1/n}$, where the measure on the right
is the $n$-fold free convolution.
\end{definition}
In dealing with infinitely divisible measures, the following
characterization, first proven in \cite{Pata}, will prove
invaluable.

\begin{theorem}\label{pataid}
Let $\{ \mu_{i,j} \}_{i \in \mathbb{N} , \ j = 1, \ldots , k_{i}}$
be an array of Borel probability measures on $\mathbb{R}$ and $\{
c_{i} \}_{i \in \mathbb{N}}$ be a sequence of real numbers.  Assume
that $\lim_{i \rightarrow \infty} \max_{j = 1,\ldots, k_{i}}
\mu_{i,j} (\{t : |t| > \epsilon \}) = 0$ for all $\epsilon > 0$ and
that the measures $\delta_{c_{i}} \boxplus \mu_{i,1} \boxplus \cdots
 \boxplus \mu_{i, k_{i}}$ converge to a probability measure $\mu$ in the
weak$^{\ast}$ topology.  Then $\mu$ is $\boxplus$-infinitely
divisible.

\end{theorem}

\begin{definition}
Let $\mu$ be a probability measure.  A decomposition $\mu = \nu
\boxplus \rho$ is said to be nontrivial if neither $\nu$ nor $\rho$
is a Dirac mass. We say that a measure $\mu$ is indecomposable if it
has no nontrivial decomposition.
\end{definition}
Such measures were studied extensively in \cite{Bel1}, \cite{Bel2}
and \cite{JC1}. We close with a theorem, first proven in \cite{BV1}
and \cite{Bel1} from which we derive a corollary that will play a
key role in the proof of Theorem \ref{k1}.

\begin{theorem}\label{belmaintheorem}
Let $\mu$ and $\nu$ be two Borel probability measures on
$\mathbb{R}$, neither of them a Dirac mass.  Then
\begin{enumerate}
\item
The point $a \in \mathbb{R} $ is an atom of the measure $\mu
\boxplus \nu$ if and only if there exist points $b,c \in \mathbb{R}$
such that $a = b + c$ and $\mu(\{ b \}) + \nu(\{ c \}) > 1$.
Moreover, $(\mu \boxplus \nu) (\{ a \}) = \mu(\{ b \}) + \nu(\{ c
\}) - 1$.

\item
The absolutely continuous part of $\mu \boxplus \nu$ is always
nonzero and its density is analytic wherever positive and finite.
More precisely, there exists an open set $U \subseteq \mathbb{R}$ so
that the density function $f(x) = \frac{d(\mu \boxplus \nu)^{ac}
(x)}{dx}$ with respect to Lebesgue measure is locally analytic on
the set U and $(\mu\boxplus\nu)^{ac}(\mathbb{R})=\int_{U} f(x) dx$.

\item
The singular continuous part of $\mu \boxplus \nu$ is zero

\end{enumerate}

\end{theorem}

\begin{corollary}\label{belcorr}
Let $\mu$ and $\nu$ be as above.  There exists a point $s \in
\mathbb{R}$ such that the cumulative distribution function
$\mathfrak{F}_{\mu \boxplus \nu}$ is continuous and increasing in a
neighborhood of $s$.

\end{corollary}
\begin{proof}

First observe that $(1)$ implies that $\mu \boxplus \nu$ has only
finitely many point masses.  To see this, assume that $a = a_{1} +
a_{2}$ and $b = b_{1} + b_{2}$ are point masses of $\mu \boxplus
\nu$ where $$(\mu \boxplus \nu)(\{ a \}) = \mu(\{ a_{1} \}) + \nu(\{
a_{2} \}) - 1$$
 $$(\mu \boxplus \nu)(\{ b \}) = \mu(\{ b_{1} \}) + \nu(\{ b_{2} \}) -
 1$$
Further assume that $a_{1} \neq b_{1}$.  This implies that $$\mu(\{
a_{1} \}) + \mu(\{ b_{1} \}) \leq 1 $$  Combined with the previous
equalities, this implies that $1 < \nu(\{ a_{2} \}) + \nu(\{ b_{2}
\})$ so that $a_{2} = b_{2}$.  This implies that, under these
assumptions, there are at most $(1 - \nu(\{ b_{2} \}))^{-1}$ point
masses of $\mu \boxplus \nu$.

Note that the nonatomic part of $\mu \boxplus \nu$ has mass strictly
greater than $0$.  To see this, let $\{x_{i} \}_{i = 1}^{n}$ be the
set of point masses of $\mu \boxplus \nu$.  Let $y$ and $\{z_{i}
\}_{i = 1}^{n}$ satisfy $y + z_{i} = x_{i}$ and $\nu (y) + \mu
(z_{i}) - 1 = (\mu \boxplus \nu )(x_{i})$ for $i = 1 , 2 , \ldots ,
n$ where these points arise as in the previous paragraph. Summing
over both sides of the equation and recalling that $\nu (y) < 1$, we
have that
$$\sum_{i = 1}^{n} (\mu \boxplus \nu) (x_{i}) = n\nu(y) - n + \sum_{i = 1}^{n}
\mu (z_{i}) < n - n + \mu(\mathbb{R}) = 1$$

Thus, for $U$ as in the previous theorem, pick an open subset $V
\subseteq U$ that contains no point masses.  This set satisfies our
claim.
\end{proof}

\section{Compactness Results for Additive Free Convolution} We
begin our investigation with a technical lemma.
\begin{lemma}\label{extensionlemma}
Let $\mu$ be a probability measure on $\mathbb{R}$.  Let $\Omega$
denote a Stolz angle on which  $F_{\mu}^{-1}$ is defined. If $\mu =
\mu_{1} \boxplus \mu_{2}$ is any decomposition of $\mu$, then
$\varphi_{\mu_{1}}$ and $\varphi_{\mu_{2}}$ have analytic extensions
to $\Omega$.  These extensions satisfy $\Im{(\varphi_{\mu_{1}}(z))}$
and $\Im{(\varphi_{\mu_{2}}(z))} \leq 0$ for all $z \in \Omega$

\end{lemma}

\begin{proof}
By assumption, $\varphi_{\mu}$ exists and is analytic on all of
$\Omega$ and, since $F_{\mu}$ increases the imaginary part,
$\varphi_{\mu} (z) \leq 0$ for all $z \in \Omega$.

Turning our attention to $\mu_{1}$ consider the subordination
function $\omega$ satisfying $F_{\mu} (z) = F_{\mu_{1}} (\omega
(z))$ for $z \in \mathbb{C}^{+}$.  Recall that
$$\lim_{y \uparrow \infty} \frac{F_{\mu} (iy)}{iy} = \lim_{y \uparrow \infty} \frac{F_{\mu_{1}} (iy)}{iy} = \lim_{y \uparrow \infty} \frac{\omega
(iy)}{iy} = 1$$

These facts imply that on a sufficiently small Stolz angle, all
three functions are invertible and we have the following:

$$ \omega \circ F_{\mu}^{-1}
= F_{\mu_{1}}^{-1}$$ Since the left hand side is defined on
$\Omega$, the right hand side must also extend to $\Omega$.  This
implies that this implies that the Voiculesu transform of $\mu_{1}$
extends to $\Omega$ and, by abuse of notation, we continue to call
this extension $\varphi_{\mu_{1}}$

With regard to the negativity the imaginary part of our analytic
extension, note that on a large enough Stolz angle, $F_{\mu_{1}}$
acts as a left inverse for $F_{\mu_{1}}^{-1}$ and $\omega \circ
F_{\mu}^{-1} = F_{\mu_{1}}^{-1}$.  Thus, $F_{\mu_{1}} (\omega (
F_{\mu}^{-1} (z) )) = z$.  As the left hand side of the equation is
defined and analytic for all $z \in \Omega$, by analytic
continuation, the same equality holds for all $z \in \Omega$.  Thus,
$$\varphi_{\mu_{1}} (z) = \omega (F_{\mu}^{-1} (z)) - z = \omega (F_{\mu}^{-1} (z))- F_{\mu_{1}} (\omega (
F_{\mu}^{-1} (z) ))$$ for all $z \in \Omega$.  As $F_{\mu_{1}}$
increases the imaginary part, our result holds.
\end{proof}

With this preliminary result out of the way, we now begin proving
tightness results.  The \textit{diameter} of a subset $\sigma
\subset \mathbb{R}$ is defined as is usual : $\textrm{diam} (\sigma)
= \sup_{x,y \in \sigma} |x - y|$.  The \textit{support} of a measure
$\mu$ (in symbols, $\textrm{supp} (\mu)$) is the complement of the
largest open $\mu$-null set.

\begin{theorem}\label{tightcompact}
Let $\mu$ be a probability measure with compact support and consider
a decomposition $\mu = \mu_{1} \boxplus \mu_{2}$.  Then
$\textrm{diam}(supp(\mu_{i})) \leq \textrm{diam}(supp(\mu))$ with
equality if and only if one of the $\mu_{i}$ is  a Dirac mass.
\end{theorem}

\begin{proof}

Consider the subordination functions $\omega_{i}$ satisfying
$F_{\mu}(z) = F_{\mu_{i}}(\omega_{i} (z))$ with Nevallina
representations:

$$\omega_{i} (z) = r_{i} + z +  \int{ \frac{1 + zt}{t-z}}
d\sigma_{i}(t)$$

$$ F_{\mu} (z) = r + z + \int \frac{1 + zt}{t-z} d\sigma(t) $$

Let $\alpha, \beta \in \mathbb{R}$ satisfy $\textrm{supp}(\mu)
\subseteq [\alpha, \beta]$, where the interval on the right side is
the smallest for which this containment holds.  Observe that
$G_{\mu}$ has a nonzero real analytic continuation across $(-\infty,
\alpha)$ so that the same must hold for $F_{\mu}$.  This implies
that $\sigma((-\infty, \alpha))  = 0$.  Since $\sigma = \sigma_{1} +
\sigma_{2}$, we also have that $\sigma_{i}(-\infty , \alpha) = 0$ so
that, by the Schwartz Lemma, $\omega_{i}$ admits analytic
continuation across $(-\infty, \alpha)$. Furthermore, $\omega_{i}$
is increasing on $(-\infty , \alpha)$ so that $F_{\mu_{i}} = F_{\mu}
\circ \omega_{i}^{-1}$ has an analytic continuation to $\omega_{i}
(-\infty , \alpha)$. This tells us that $\textrm{supp}(\mu_{i})
\subset \mathbb{R} \setminus \omega_{1} ((-\infty, \alpha))$.

Now, observe that $\omega_{i} (x) - x \rightarrow r_{i} - m_{i}$ as
$x \rightarrow \pm \infty$ where $m_{i}$ is the first moment of
$\sigma_{i}$. Differentiating the Nevanlinna representation of
$\omega_{i}$, it is clear that $\omega'_{i}(x) \geq 1$ for $x <
\alpha$. Thus,
$$\omega_{i} (\alpha - \epsilon) = \int_{x}^{\alpha - \epsilon} \omega'_{i}(t)dt + \omega_{i}(x) \geq \int_{x}^{\alpha - \epsilon} dt + x + ( \omega_{i} (x) - x) \rightarrow \alpha - \epsilon + r_{i} - m_{i}$$
It follows that $ (-\infty , \alpha + r_{i} - m_{i}) \subseteq
\omega_{i} ((-\infty , \alpha))$. Similarly, $(\beta + r_{i} -
m_{i}, \infty) \subseteq \omega_{i} (\beta , \infty) $. These two
observations imply that $\textrm{supp} (\mu_{i}) \subseteq [\alpha +
r_{i} - m_{1} , \beta + r_{i} - m_{i}]$.  Hence, we have that
$\textrm{diam}(\textrm{supp}(\mu_{i})) \leq
\textrm{diam}(\textrm{supp}(\mu))$.

With regard the equality claim, observe that our measure $\sigma_{1}
= 0$ implies that $\mu_{1}$ is a translation of $\mu$.  This implies
that $\mu_{2}$ is a Dirac mass.    Thus, by assuming that neither
$\mu_{1}$ nor $\mu_{2}$ is a Dirac mass, we have that $\sigma_{i}
\neq 0$ for $i = 1,2$.  This implies that $\omega_{i}'(t) > 1$ for
$t < \alpha$.  It follows that $\textrm{supp} (\mu_{i}) \subsetneqq
[\alpha + r_{i} - m_{1} , \beta + r_{i} - m_{i}]$, and our claim
follows.
\end{proof}

In what follows, for $O \subset \mathbb{R}$, we let
$\textrm{conv}(O)$ be the smallest interval containing the set $O$.

\begin{lemma}\label{compact2}
Let $\mu_{1}$ and $\mu_{2}$ be probability measures with compact
support.  Then $\textrm{supp} (\mu_{1} \boxplus \mu_{2}) \subseteq
\textrm{conv} (\textrm{supp} (\mu_{1}) + \textrm{supp} (\mu_{2}))$.
\end{lemma}
\begin{proof}
Let $x_{1}$ and $x_{2}$ be freely independent random variables in a
$W^{*}$ probability space $(A, \tau)$ with respective distributions
$\mu_{1}$ and $\mu_{2}$. Let $c_{i} = \inf \{ t \in \sigma (x_{i})
\}$ and $d_{i} = \sup \{ t \in \sigma (x_{i}) \}$. It is precisely
the content of Theorem $4.16$ in \cite{BV1} that $x_{1} -c_{1}I +
x_{2} - c_{2}I$ is a positive random variable.  Thus, its spectrum
is contained in the positive real numbers. Since the spectrum of a
self-adjoint operator contains the support of its distribution, we
have that the distribution of $x_{1} -c_{1}I + x_{2} - c_{2}I$ is
supported in the positive reals. Similarly, the distribution of
$x_{1} -d_{1}I + x_{2} - d_{2}I$ is supported in the negative reals.
Thus, $\textrm{supp}(\mu_{1} \boxplus \mu_{2}) \subseteq [c_{1} +
c_{2} , d_{1} + d_{2}]$, which is equivalent to our claim.
\end{proof}

We now extend the above theorem to measures with unbounded support.
For a measure $\mu$, recall that $\mathfrak{F}_{\mu}$ denotes its
cumulative distribution function.  We shall let
$\Omega_{\epsilon}(\mu) = \{t\in \mathbb{R} : \epsilon <
\mathfrak{F}_{\mu} (t) < 1- \epsilon \}$.

\begin{theorem}\label{noncompact1}
Let $\mu = \mu_{1} \boxplus \mu_{2}$.  For $\epsilon > 0$ we have
that $\overline{\Omega_{\epsilon} (\mu)} \subseteq
\overline{\Omega_{\epsilon/2} (\mu_{1})} +
\overline{\Omega_{\epsilon/2} (\mu_{2})}$.

\end{theorem}
\begin{proof}

Let $a_{i}$ and $b_{i}$ denote the left and right endpoints of
$\Omega_{\epsilon / 2} (\mu_{i})$. Consider the probability measures
$\mu_{i,\epsilon / 2}$ defined as follows:
$$\mu_{i,\epsilon / 2} (\sigma) = \mu_{i} (\sigma \cap \Omega_{\epsilon / 2} (\mu_{i})) + \left(\frac{1 - \mu_{i} (\Omega_{\epsilon / 2} (\mu_{i}) )}{2}\right) \delta_{a_{i}}(\sigma)  + \left(\frac{1 - \mu_{i} (\Omega_{\epsilon / 2} (\mu_{i}) )}{2}\right) \delta_{b_{i}}(\sigma)$$

Observe that $d_{\infty} (\mu_{i} , \mu_{i, \epsilon / 2}) \leq
\epsilon / 2$ where $d_{\infty}$ denotes the Kolmogorov metric.
Further observe that $\textrm{supp}(\mu_{i, \epsilon / 2}) =
\Omega_{\epsilon / 2} (\mu_{i})$. It follows that

$$d_{\infty} (\mu , \mu_{1,\epsilon/2} \boxplus \mu_{2,\epsilon/2}) \leq d_{\infty} (\mu_{1} , \mu_{1, \epsilon / 2}) + d_{\infty} (\mu_{2} , \mu_{2, \epsilon / 2}) \leq \epsilon$$

Observe that $\mathfrak{F}_{\mu} (t) \in (\epsilon , 1- \epsilon)$
implies that $\mathfrak{F}_{\mu_{1,\epsilon/2} \boxplus
\mu_{2,\epsilon/2}} (t) \in (0,1)$. Thus, $\Omega_{\epsilon}(\mu)
\subseteq \textrm{supp}(\mu_{1,\epsilon/2} \boxplus
\mu_{2,\epsilon/2})$.  By Lemma \ref{compact2}, we have that
$\textrm{supp}(\mu_{1,\epsilon/2} \boxplus \mu_{2,\epsilon/2}))
\subset \textrm{conv} (\textrm{supp}(\mu_{1,\epsilon/2}) +
\textrm{supp}(\mu_{2,\epsilon/2})) =
\textrm{conv}(\overline{\Omega_{\epsilon/2} (\mu_{1})} +
\overline{\Omega_{\epsilon/2} (\mu_{2})})$.
\end{proof}

We close with the main result of the section.  Observe that this
theorem lacks the quantitative information found in Theorem
\ref{noncompact1}. The hope was to extend Theorem \ref{tightcompact}
in a similar manner, but such an approach proved elusive.  We have
found no negative results in this direction so we conjecture that
$\Omega_{\epsilon}(\mu_{i}) \subseteq x + \Omega_{\epsilon/2} (\mu)$
for some $x \in \mathbb{R}$.  However, the theorem below provides us
with tightness and will suffice for the applications that follow.

Let $\nu$ be a measure satisfying $0 < \nu(\mathbb{R}) \leq 1$. We
extend the definition of the Cauchy and $F$-transform by letting
$G_{\nu}(z) = \int_{\mathbb{R}} (z -t)^{-1} d\nu(t)$ and $F_{\nu(z)}
= 1 / G_{\nu(z)}$.  Observe that for $\lambda = \nu(\mathbb{R})$,
the measure $\hat{\nu} = \lambda^{-1} \nu$ is in fact a probability
measure. This provides us with the following inequality which we
shall exploit in what follows.
$$\Im F_{\nu(z)} = \lambda^{-1} \Im F_{\hat{\nu}}(z) \geq \lambda^{-1}\Im(z) $$

\begin{theorem}\label{noncompact2}
Let $\mu = \mu_{1,k} \boxplus \mu_{2,k}$ for all $k \in \mathbb{N}$.
Then there exist translations $\{ \hat{\mu}_{i,k} \}$ so that $\mu =
\hat{\mu}_{1,k} \boxplus \hat{\mu}_{2,k}$ and the family of measures
$\{ \hat{\mu}_{i,k} \}$ is tight for $i = 1,2$.
\end{theorem}

Before embarking on the proof, we remark that there are two ways for
tightness to fail.  The first is to take an otherwise tight sequence
of measures and translate their support to $\pm \infty$.  The second
is if the mass of your measures becomes more spread out. Since our
theorem assumes away the former case, the idea of the proof is to
show that the latter cannot happen.  We quantify the latter case as
follows: a sequence of measures $\{ \mu_{k} \}_{k \in \mathbb{N}}$
cannot be translated to tightness if and only if there exists a
$\gamma \in [0,1)$ such that $\liminf_{k} \sup_{t \in \mathbb{R}}
(\mu_{k} (t-a, t+a)) < \gamma $ for all $a \in \mathbb{R}^{+}$.

\begin{proof}
Assume that $\{ \mu_{1,k} \}_{k \in \mathbb{N}}$ is tight which is
equivalent to sequential precompactness in the weak$^{\ast}$
topology. As we established in Lemma \ref{extensionlemma},
$F_{\mu}^{-1}$, $F_{\mu_{1,k}}^{-1}$ and $F_{\mu_{2,k}}^{-1}$ extend
to a common domain for all $k$, which we shall denote by $\Omega$ in
what follows, on which they satisfy $F_{\mu}^{-1} (z) -
F_{\mu_{1,k}}^{-1}(z) + z = F_{\mu_{2,k}}^{-1}(z)$. Recall that,
according to Lemma \ref{BVconv}, weak convergence is equivalent to
the uniform convergence of the functions $F_{\mu_{1,k}}^{-1}$ on
compact subsets of a Stolz angle $\Gamma_{\alpha, \beta}$ to a
function $F$ satisfying $F(iy)/iy \rightarrow 1$ as $y \rightarrow
\infty$. The equation above implies that $F_{\mu_{2,k}}^{-1}$ is
similarly behaved on $\Gamma_{\alpha, \beta}$ so that $\{ \mu_{2,k}
\}$ is also weakly convergent along this subsequence.  Thus, $\{
\mu_{1,k} \}$ is tight implies the same for $\{ \mu_{2,k} \}$.

With that in mind, we may assume, for the sake of contradiction,
that the family $\{ \mu_{1,k} \}$ cannot be translated to form a
tight family of measures along any subsequence.  This implies that
there exists a $\gamma \in (0,1)$ such that $\liminf_{k} \sup_{x \in
\mathbb{R}} (\mu_{1,k} (-a + x, a + x)) < \gamma < 1$ for all $a \in
\mathbb{R}^{+}$.  Passing to subsequences and possibly renumbering
our measures, we may assume that $\sup_{x \in \mathbb{R}} (
\mu_{1,k}(x - k, x + k)) < \gamma$.

Now, pick $\epsilon > 0$ such that $(1 - \epsilon) > \gamma$.  Let
$w = ib$ where $b \in \mathbb{R}^{+}$ is chosen so that $w \in
\Omega$ and $|F_{\mu}^{-1} (w) - w| \leq \epsilon|w| = \epsilon b$.
Observe that
$$F_{\mu}^{-1} (w) = F_{\mu_{1,k}}^{-1}(w) +  F_{\mu_{2,k}}^{-1}(w) -
w$$ implies that $$\Im{F_{\mu_{1,k}}^{-1}(w)} +
\Im{F_{\mu_{2,k}}^{-1}(w)} \geq b(2 - \epsilon)$$ In Lemma
\ref{extensionlemma}, we showed that $F_{\mu_{i,k}}^{-1}$ decreases
the imaginary part so that $$\Im{F_{\mu_{1,k}}^{-1}(ib)} \geq b(1 -
\epsilon)$$ Further observe that analytic continuation implies that
$F_{\mu_{1,k}}(F_{\mu_{1,k}}^{-1} (z)) = z$ for all $z \in \Omega$
so that, in particular, $ib = F_{\mu_{1,k}}(F_{\mu_{1,k}}^{-1} (ib))
$.

Now, let $z_{k} =  F_{\mu_{1,k}}^{-1}(ib)$ and denote by $t_{k}$ the
real part of this number (the real part can vary as wildly as you
would like but we will show that this is not a problem).    We
decompose $\mu_{1,k}$ so that $\mu_{1,k} = \nu_{1,k} + \rho_{1,k}$
where $\nu_{1,k} (\mathbb{R}) = \lambda_{k} < \gamma$ and
$\rho_{1,k} ([t_{k} - k , t_{k} + k]) = 0$.  A decomposition with
these properties exists because of the fact that $\sup_{x \in
\mathbb{R}} ( \mu_{1,k}(x - k, x + k)) < \gamma$. We will use the
last of the above properties to show that $|F_{\mu_{1,k}}(z_{k}) -
F_{\nu_{1,k}} (z_{k})| \rightarrow 0$. We will then use the fact
that $F_{\nu_{1,k}}$ increases imaginary part in proportion to
$\lambda_{k}^{-1}$ to derive a contradiction.

Observe that
$$F_{\mu_{1,k}} (z_{k}) = \frac{1}{G_{\nu_{1,k}}(z_{k}) + {G_{\rho_{1,k}} (z_{k})} }$$
and that

$$|{G_{\rho_{1,k}} (z_{k})}| = |\int_{\mathbb{R} \setminus (t_{k}-k, t_{k} + k)} \frac{1}{z_{k}-t} d \rho_{1,k}(t)| \rightarrow 0 $$
as $k \rightarrow \infty$.  This second fact is clear since
$\rho_{1,k}$ is a subprobability measure and, since $\Re{(z_{k})} =
t_{k}$, the above integrand converges to $0$ uniformly on the domain
of integration as $k \uparrow \infty$. Now, if $\liminf_{k}
|G_{\nu_{1,k}}(z_{k})| = 0$ then $\limsup_{k} |F_{\mu_{1,k}}
(z_{k})| = \infty$ which would contradict the fact that
$F_{\mu_{1,k}} (z_{k}) \equiv ib$. Thus, we may assume that
$|G_{\nu_{1,k}}(z_{k})| \geq c
> 0$.  This implies that $\lambda_{k} > 0$.

Consider the quantity
$$|F_{\mu_{1,k}} (z_{k}) - F_{\nu_{1,k}} (z_{k})| =  |((G_{\nu_{1,k}}
(z_{k}) - G_{\mu_{1,k}} (z_{k})))(G_{\mu_{1,k}} (z_{k})
G_{\nu_{1,k}} (z_{k}))^{-1}| $$ Observe that the numerator of the
right hand side goes to zero since $G_{\mu_{1,k}} - G_{\nu_{1,k}} =
G_{\rho_{1,k}}$ and the denominator is bounded away from zero since
$|G_{\nu_{1,k}}(z_{k})| \geq c
> 0$ and $|G_{\mu_{1,k}}(z_{k})| \equiv b^{-1} > 0$.  Thus, $|F_{\mu_{1,k}} (z_{k}) - F_{\nu_{1,k}}
(z_{k})| \rightarrow 0$ as $k \uparrow \infty$.

Recalling the remarks preceding this theorem, we consider the
probability measure $\hat{\nu}_{1,k} = \lambda_{k}^{-1}\nu_{1,k}$ so
that $F_{\nu_{1,k}} (z_{k}) = \lambda_{k}^{-1} F_{\hat{\nu}_{1,k}}
(z_{k})$. We then have

$$b = \Im{F_{\mu_{1,k}} (z_{k})} = \lim_{k \uparrow \infty}\Im{F_{\mu_{1,k}} (z_{k})} =  \lim_{k \uparrow \infty} \lambda_{k}^{-1}
\Im{F_{\hat{\nu}_{1,k}} (z_{k})} $$ $$\geq \lim_{k \uparrow \infty}
\lambda_{k}^{-1} \Im{(z_{k})}   \geq \gamma^{-1}(b(1 - \epsilon)) >
b$$ This contradiction completes our proof.
\end{proof}

We end with a few remarks and corollaries.  We single out the
following fact from last theorem for easy reference.

\begin{corollary}\label{noncompact2easycorr}
Let $\mu = \mu_{1,k} \boxplus \mu_{2,k}$ be a family of
decompositions.  Assume that $\{ \mu_{1,k} \}_{k \in \mathbb{N}}$ is
tight.  Then $\{ \mu_{2,k} \}_{k \in \mathbb{N}}$ is tight.
\end{corollary}

As we stated before the proof of Theorem \ref{noncompact2}, a family
of measures can fail to be tight either by being translated to $\pm
\infty$ or by becoming more spread out.  For $t \in (0,1)$, we shall
say that a measure $\mu$ is $\textit{t-centered}$ if
$\mathfrak{F}_{\mu}(s) < t$ for $s < 0$ and $\mathfrak{F}_{\mu}(s)
\geq t$ for $s \geq 0$.  Right continuity of the distribution
function implies that a measure $\mu$ has a unique $t$-centered
translation. Observe that when $t = 1/2$, $t$-centered is simply the
more familiar median $0$.  Now, if we assume that we have a family
of decompositions $\mu = \mu_{1,k} \boxplus \mu_{2,k}$ where $\{
\mu_{1,k} \}_{k \in \mathbb{N}}$ are assumed to be $t$-centered,
then the supports of these measures are not being sent to $\infty$.
By \ref{noncompact2}, we have the following corollary.

\begin{corollary}
Let $\mu = \mu_{1,k}\label{tcenter} \boxplus \mu_{2,k}$ where $\{
\mu_{1,k} \} $ are $t$-centered where $t$ is allowed to range over a
compact subset of $(0,1)$. Then $\{ \mu_{i,k} \}_{k \in \mathbb{N}}$
forms a tight family.
\end{corollary}

The following variation will prove useful in what follows.

\begin{corollary}\label{noncompact2corrolary}
Let $\{ \mu_{n} \}_{n \in \mathbb{N}}$ be a tight sequence of
measures.  Assume that to each member of this family we associate a
family of decompositions $\mu_{n} = \nu_{n,k} \boxplus \rho_{n,k}$
for $k \in \mathbb{N}$.  Then we may translate our measures so as to
form  tight families $\{ \hat{\nu}_{n,k} \}_{n,k \in \mathbb{N}}$
and $\{ \hat{\rho}_{n,k} \}_{n,k \in \mathbb{N}}$ with the property
that $\mu_{n} = \hat{\nu}_{n,k} \boxplus \hat{\rho}_{n,k}$ for all
$n,k \in \mathbb{N}$.
\end{corollary}
\begin{proof}

We assume that each $\nu_{n,k}$ has median $0$.

Assume that $\{ \nu_{n(i), k(i)} \}_{i \in \mathbb{N}}$ has no
subconvergent sequence.  Let $\mu$ be a cluster point of $\{
\mu_{n(i)} \}_{i \in \mathbb{N}}$.  By Lemmas \ref{tightdomain} and
\ref{extensionlemma}, we have that there exists a truncated cone
$\Gamma_{\alpha , \beta}$ so that for $i$ large enough,
$F_{\mu}^{-1} , F_{\mu_{n(i)}}^{-1} , F_{\nu_{n(i) , k(i)}}^{-1}$
and $F_{\rho_{n(i) , k(i)}}^{-1}$ are all defined and satisfy
$F_{\nu_{n(i) , k(i)}}^{-1}(z) + F_{\rho_{n(i) , k(i)}}^{-1}(z) - z
= F_{\mu_{n(i)}}^{-1} (z) \rightarrow F_{\mu}^{-1}(z)$ uniformly
over compact subsets of $\Gamma_{\alpha , \beta}$.

Now, since we have centered our measures $\nu_{n(i),k(i)}$ by
assuming median $0$, the lack of a convergent subsequence amounts to
assuming that $$\liminf_{i}( \sup_{t\in \mathbb{R}} \nu_{n(i) ,
k(i)} ([t -a , t + a]) ) \rightarrow \gamma < 1$$ for all $a \in
\mathbb{R^{+}}$.  At this point, one need only observe that every
step of the proof of Theorem \ref{noncompact2} holds under the
weaker assumption that $F_{\nu_{n(i) , k(i)}}^{-1}(z) +
F_{\rho_{n(i) , k(i)}}^{-1}(z) - z \rightarrow F_{\mu}^{-1}(z)$ as
opposed to assuming outright equality.  This completes our proof.
\end{proof}

\section{A Khintichine Decomposition for Additive Free
Convolution}

\begin{lemma}\label{diff_bound}
Let $\{ \mu_{i} \}_{i \in I}$ be a tight family of probability
measures. Then, for every $C > 0$, there exists a Stolz angle
$\Gamma_{\alpha , \beta}$ such that $|\varphi'_{\mu_{i}} (z)| \leq
C|z|$ for all $z\in \Gamma_{\alpha , \beta}$ and $n \in \mathbb{N}$
\end{lemma}
\begin{proof}
It was shown in the proof of Theorem $5.2$ in \cite{BV1} that, given
a tight family of measures $\{ \mu_{i} \}_{i \in I}$, there exists
an $\alpha > 0$ such that $F_{\mu_{i}} = z(1 + o(1))$ uniformly as
$|z| \uparrow \infty$ for $z\in \Gamma_{\alpha, 0}$.  Thus, for
fixed $C > 0$, we may find a  $\beta$ large enough so that
$|\varphi_{\mu_{i}} (z)| = |F_{\mu_{i}}^{-1} (z) - z| \leq C|z|$ for
$z \in \Gamma_{\alpha, \beta}$ , $n \in \mathbb{N}$. By Cauchy's
formula,
$$ |\varphi_{\mu_{i}}'(z)| = (2\pi)^{-1} |\int_{|\zeta| = 1}{\frac{\varphi_{\mu_{i}}(z)}{(\zeta - z)^{2} } d \zeta} | \leq C|z|$$
\end{proof}

We will now define the functional that will be the main tool in the
proof of our main theorems.  Let $\mu$ be a probability measure. Let
$M_{0}$ be the set of all median $0$ probability measures $\nu$
satisfying $\mu = \nu \boxplus \rho$ for some probability measure
$\rho$.  It is a consequence of Corollary \ref{tcenter} that this is
a tight family of measures. Let $\Gamma_{\alpha , \beta}$ be a Stolz
angle on which $F_{\mu}^{-1}$ is defined and for which Lemma
\ref{diff_bound} is satisfied with regard to $M_{0}$. Consider the
set $\Gamma '  = \{ z \in \mathbb{C}^{+} : \alpha + 1 > \Im(z) >
\alpha , \ \Im(z) > \beta \Re(z) \} \subset \Gamma_{\alpha , \beta}$
and let $M_{\Gamma'}$ be the set of probability measures $\nu$ such
that $\varphi_{\nu}$ has analytic extension to $\Gamma'$ such that
$\Im \varphi_{\nu} (z) \leq 0$ for all $z \in \Gamma'$. For $\nu \in
M_{\Gamma'}$, let $\Lambda (\nu):= -\int_{\Gamma'} \Im \varphi_{\nu
(z)} dA(z)$ where $A$ denotes the area measure.

Observe that, by Lemma \ref{id}, for any decomposition $\mu = \rho
\boxplus \nu$ we have that $\rho , \nu \in M_{\Gamma'}$.
Furthermore, we claim the following properties for our functional
$\Lambda$.
\begin{enumerate}

\item
$\Lambda$ is weakly continuous.

\item
$\Lambda(\nu \boxplus \rho) = \Lambda(\nu) + \Lambda(\rho)$ for all
$\nu , \rho \in M_{\Gamma'}$.

\item
$0 \leq \Lambda(\nu) < \infty$ for all $\nu \in M_{\Gamma'}$.
$\Lambda (\nu) = 0$ if and only if $\nu$ is a Dirac mass.

\item
$\Lambda(\nu \boxplus \delta_{t}) = \Lambda(\nu)$ for all $t \in
\mathbb{R}$ and $\nu \in M_{\Gamma'}$.
\end{enumerate}
The only fact that requires argument is that $\Lambda(\nu) = 0$ if
and only if $\nu$ is a Dirac mass.  One direction is clear since the
Voiculescu transform of a Dirac mass is simply a real constant.
Furthermore, since $-\Im(\varphi_{\nu} (z)) \geq 0$ for all $z \in
\Gamma'$, $\Lambda(\nu) = 0$ implies that $-\Im(\varphi_{\nu} (z))
\equiv 0$ for $z \in \Gamma'$.  Analytic continuation implies that
$\varphi_{\nu}$ is a real constant which implies that $\nu$ is a
Dirac mass.

\begin{theorem}\label{id}
Let $\mu$ be a probability measure with the property that for every
non-trivial decomposition $\mu = \mu_{1} \boxplus \mu_{2}$, neither
$\mu_{1}$ nor $\mu_{2}$ is indecomposable.  Then $\mu$ is infinitely
divisible.
\end{theorem}
\begin{proof}

We first note that for every $\epsilon > 0$, there exists a
decomposition $\mu = \mu_{1} \boxplus \mu_{2}$ such that $0< \Lambda
(\mu_{1}) < \epsilon$.  Assume otherwise and let $\alpha > 0$ be the
infimum of $\Lambda$ over all non-trivial decompositions of $\mu$.
By Theorem \ref{noncompact2}, there exists a sequence of
decompositions $\mu = \mu_{1,k} \boxplus \mu_{2,k}$ so that the
families $\{ \mu_{i,k} \}_{k = 1}^{\infty}$ are tight and so that
$\Lambda (\mu_{1,k}) \rightarrow \alpha$.  Taking weak cluster
points $\mu_{1}$ and $\mu_{2}$, by weak continuity of both $\Lambda$
and $\boxplus$ we have that $\mu = \mu_{1} \boxplus \mu_{2}$ and
$\Lambda (\mu_{1}) = \alpha$.  By assumption, $\mu_{1}$ has a
nontrivial decomposition $\mu_{1} = \nu_{0} \boxplus \nu_{1}$. Since
neither component is a Dirac mass, we have that $\alpha > \Lambda
(\nu_{i})
> 0$ so that the decomposition $\mu = \nu_{0} \boxplus (\nu_{1}
\boxplus \mu_{2})$ violates minimality of $\alpha$.

We now claim that for every $t \in (0,\Lambda (\mu))$ there exists a
decomposition $\mu = \mu_{1} \boxplus \mu_{2}$ such that $\Lambda
(\mu_{1}) = t$.  To see this, let $\alpha$ be the supremum of all
values of $\Lambda (\mu_{1})$ that are $\leq t$.  The previous
paragraph implies that $ \alpha > 0$.  We again take a sequence $\mu
= \mu_{1,k} \boxplus \mu_{2,k}$ so that $\Lambda (\mu_{1,k})
\uparrow \alpha$ so that the cluster points $\mu_{i}$ satisfy $\mu =
\mu_{1} \boxplus \mu_{2}$ and $\Lambda (\mu_{1}) = \alpha$.  If
$\alpha < t$, by the above argument, we can break a chunk of size
less than $t- \alpha$ from $\mu_{2}$ so as to attain a
contradiction.  Thus, $\Lambda$ takes values on all of $(0,\Lambda
(\mu))$ as it ranges over divisors of $\mu$.

By induction, for every $n \in \mathbb{N}$ we can find a
decomposition $\mu = \mu_{n,1} \boxplus \ldots \boxplus \mu_{n,n}
\boxplus \delta_{x_{n}}$ such that $\Lambda (\mu_{n,i}) =
\Lambda(\mu)/n $ and $\mu_{n,i}$ has median $0$ for all $i = 1,
\ldots , n$.  The real number $x_{n}$ is the shift constant that
necessarily arises when centering these measures.  We now claim that
the array $\{ \mu_{n,j} \}_{n \in \mathbb{N} , j = 1, \ldots , n}$
converges to $\delta_{0}$ uniformly as $n \uparrow \infty$.

Observe that Corollary \ref{tcenter} implies that our array is
tight. Let $\nu$ be any cluster point and let $\{ \mu_{k_{n} ,
j_{n}} \}_{n \in \mathbb{N}}$ be a subsequence converging to $\nu$.
By Lemma \ref{BVconv}, $\varphi_{\{ \mu_{k_{n} , j_{n}} \}} (z)
\rightarrow \varphi_{\nu} (z)$ uniformly on compact subsets of a
Stolz angle $\Gamma^{\ast} \subseteq \Gamma$.  Now, observe that
$\Gamma'$ and $\Gamma^{\ast}$ may be disjoint. However, there exist
$a,b \in \mathbb{R}$ such that $ia \in \Gamma'$ and $ib \in
\Gamma^{\ast}$.

Observe that $\varphi_{\mu_{k_{n} , j_{n}}}$ is a normal family on
$\Gamma' \cup i[a,b]$, which implies precompactness.  By analytic
continuation, any cluster point must agree with $\varphi_{\nu}$ on
$i[a,b] \cap \Gamma^{\ast}$. This implies that $\varphi_{\nu}$ has
analytic continuatin to $\Gamma'$ that satisfies $\varphi_{\nu}(z) =
\lim_{n \uparrow \infty} \varphi_{\mu_{k_{n} , j_{n}}}(z)$ for $z
\in \Gamma'$.  Now, observe that the fact that $\Lambda(\mu_{k_{n} ,
j_{n}}) \rightarrow 0$ implies that $-\int_{\Gamma'}
\varphi_{\mu_{k_{n} , j_{n}}}(z) dA(z) \rightarrow 0$.  By Lemma
\ref{diff_bound}, we have a bound on the derivatives of these
functions so that, recalling that the imaginary parts of these
functions are negative, $\Im \varphi_{\mu_{k_{n} , j_{n}}}(z)
\rightarrow 0$ for $z \in \Gamma'$. This implies that $\Im
\varphi_{\nu} (z) = 0$ for $z \in \Gamma'$. Thus, $\nu$ is a dirac
mass and our median $0$ assumption implies that $\nu = \delta_{0}$

Thus, our array is tight and every subsequence converges to
$\delta_{0}$.  This implies that our array converges to $\delta_{0}$
uniformly over $n$. By Theorem \ref{pataid}, $\mu$ is infinitely
divisible.

\end{proof}

\begin{lemma}\label{bellemma}
Let $\{ \mu_{n} \}_{n \in \mathbb{N}}$ be a sequence of $t$-centered
measures that converge weakly to $\mu$.  Assume that for $s \in
\mathbb{R}$ such that $\mathfrak{F}_{\mu} (s) = t$, we have that
$\mathfrak{F}_{\mu}$ is continuous and strictly increasing in a
neighborhood $s$. Then $s = 0$ or, in other words, $\mu$ is
$t$-centered.
\end{lemma}
\begin{proof}
Choose $\epsilon > 0$ such that $\mathfrak{F}_{\mu}$ is continuous
on $(s - 2\epsilon , s + 2\epsilon)$.  Let $0 < \epsilon' <
\epsilon$ and observe that utilizing the L\'evy metric and our
assumption of weak convergence, we have the following inequality for
$n$ large enough, independent of $\epsilon$:
$$\mathfrak{F}_{\mu} (s - \epsilon - \epsilon') - \epsilon' \leq \mathfrak{F}_{\mu_{n}}(s - \epsilon) \leq \mathfrak{F}_{\mu}(s - \epsilon + \epsilon') + \epsilon' $$
By continuity of $\mathfrak{F}_{\mu}$ at these points, it follows
that $\mathfrak{F}_{\mu_{n}}(s - \epsilon) \rightarrow
\mathfrak{F}_{\mu}(s - \epsilon)$.  Similarly
$\mathfrak{F}_{\mu_{n}}(s + \epsilon) \rightarrow
\mathfrak{F}_{\mu}(s + \epsilon)$.  Thus, for $n$ large enough, we
have that $\mathfrak{F}_{\mu_{n}}(s - \epsilon) < t$ and
$\mathfrak{F}_{\mu_{n}}(s - \epsilon) > t$.  This implies that $0
\in (s - \epsilon , s + \epsilon)$.  As $\epsilon$ was arbitrary,
this implies that $s = 0$.
\end{proof}

It is clear from the statement of the previous lemma that it will be
used in conjunction with Corollary \ref{belcorr}.  Indeed, it is
precisely the content of this corollary that measures with
non-trivial decompositions satisfy the hypotheses in Lemma
\ref{bellemma}, which will play a small but key role in the proof of
the following theorem.

\begin{theorem}\label{k1}
Let $\mu$ be a probability measure.  Then there exist measures
$\mu_{i}$ with $i = 0,1,2, \ldots $ such that $\mu_{0}$ is
$\boxplus$-infinitely divisible, $\mu_{i}$ is indecomposable for $i
= 1, 2, \ldots $, and $\mu = \mu_{0} \boxplus \mu_{1} \boxplus
\mu_{2} \boxplus \cdots$. This decomposition is not unique.
\end{theorem}
\begin{proof}

If $\mu$ is infinitely-divisible, then we are done.  If not, by
Theorem \ref{id}, $\mu$ has non-trivial divisors. Otherwise, let
$\alpha_{0} = \sup \{ \Lambda(\rho) \}$ where the supremum is taken
over all indecomposable probability measures $\rho$ satisfying $\mu
= \nu \boxplus \rho$ for some probability measure $\nu$. Let
$\mu_{1}$ be chosen so that  $\mu = \mu_{0,1} \boxplus \mu_{1}$,
$\Lambda (\mu_{1}) > \alpha_{0} / 2$ and $\mu_{1}$ is
indecomposable.  By translating our measures, $\mu_{1}$ is assumed
to be $t$-centered for a $t$ to be chosen later (for the real number
$s$ such that $\mu_{1} \boxplus \delta_{s}$ is $t$-centered, we need
only consider the decomposition $\mu = (\mu_{0,1} \boxplus
\delta_{-s}) \boxplus (\mu_{1} \boxplus \delta_{s})$ and all of the
relevant properties will be satisfied).

At the $n$th stage of this process, we let $\alpha_{n-1} = \sup \{
\Lambda(\rho) \}$ where the supremum is taken over all
indecomposable probability measures $\rho$ satisfying $\mu_{0 , n-1}
= \nu \boxplus \rho$ for some measure $\nu$ (unless $\mu_{0,n-1}$ is
infinitely divisible, at which point we are done). We then let
$\mu_{n}$ be chosen such that $\mu_{0,n-1} = \mu_{0,n} \boxplus
\mu_{n}$, $\Lambda(\mu_{n}) > \alpha_{n} / 2$, and $\mu_{n}$ is
indecomposable. By translating $\mu_{0,n}$ and $\mu_{n}$, we may
further assume that $\mu_{1} \boxplus \cdots \boxplus \mu_{n}$ is
$t$-centered.  If at any point $\alpha_{n} = 0$, then by Theorem
\ref{id}, we are done. We therefore assume that $\alpha_{n}
> 0$ for all $n \in \mathbb{N}$

In what follows, we utilize the following notation:
$$ \nu_{n} = \mu_{1} \boxplus \cdots \boxplus \mu_{n} $$
$$ \nu_{n,m} = \mu_{m+1} \boxplus \cdots \boxplus \mu_{n}$$
$$ \nu_{\infty,m} =  \lim_{n \uparrow \infty}\mu_{m+1} \boxplus \cdots \boxplus \mu_{n}$$
where we will show at a latter point that the latter actually
converges.

Note that Corollary \ref{noncompact2corrolary} implies that
$\{\nu_{n,m} \}_{n,m \in \mathbb{N}} $ is a tight family.  It
follows that $\{ \nu_{n} \}_{n \in \mathbb{N}}$ is also tight.  We
now claim that this sequence of measures is actually convergent for
an appropriate choice of $t$ in the sense of $t$-centeredness.

Proceeding with our claim, observe that  $\Lambda(\mu) =
\Lambda(\mu_{0,n}) + \Lambda(\nu_{n}) = \Lambda(\mu_{0,n}) +
\Lambda(\nu_{m}) + \Lambda(\nu_{n,m})$ for all $m < n \in
\mathbb{N}$. Observe that $\Lambda(\mu_{0,n})$ is bounded and
decreasing so necessarily converges.  This implies that
$\Lambda({\nu_{n,m}})$ represents the tail of a convergent series
and must therefore go to $0$ uniformly as $m \uparrow \infty$ (note
that this implies that $\alpha_{n} \rightarrow 0$).

Let $\hat{\nu}_{n,m}$ be the translation of $\nu_{n,m}$ with median
$0$ and observe that $\Lambda(\nu_{n,m}) = -\int_{\Gamma'}
\Im\varphi_{\nu_{n,m}}(z)dA(z) = -\int_{\Gamma'}
\Im\varphi_{\hat{\nu}_{n,m}}(z)dA(z) = \Lambda(\hat{\nu}_{n,m})$. By
Lemma \ref{diff_bound}, $\varphi'_{\hat{\nu}_{n,m}}$ is bounded on
$\Gamma'$. Since $\Lambda (\hat{\nu}_{n,m}) \rightarrow 0$ as $m
\uparrow \infty$, we have that $-\Im\varphi_{\hat{\nu}_{n,m}}(z)
\rightarrow 0$ uniformly over $\Gamma'$ as $m \uparrow \infty$.  By
Lemma \ref{BVconv}, any cluster point $\hat{\nu}$ of $\{
\hat{\nu}_{n,m} \}_{n,m \in \mathbb{N}}$ must satisfy
$\varphi_{\nu}(z) = 0$ for $z \in \Gamma'$.  This implies that
$\hat{\nu}$ is a Dirac mass.  Thus, any cluster point $\nu$ of $\{
\nu_{n,m}\}$ as $m \uparrow \infty$ must also be a Dirac mass.

Thus, the set of cluster points of $\{ \nu_{n} \}_{n \in
\mathbb{N}}$ is of the form $\{ \rho \boxplus \delta_{r} \}_{r \in
K}$ where $K$ is a compact subset of $\mathbb{R}$.  Since we are
assuming that $\alpha_{n} > 0$ for all $n \in \mathbb{N}$, we have
that $\rho = \mu_{1} \boxplus \nu$ where $\nu$ is some non-trivial
cluster point of $\{ \nu_{n,1}  \}_{n \in \mathbb{N}}$.  In
particular, $\rho$ has a non-trivial decomposition so that, by
Corollary \ref{belcorr}, there exist points $s \in \mathbb{R}$ and
$t \in (0,1)$ such that $\mathfrak{F}_{\rho}(s) = t$ and
$\mathfrak{F}_{\rho}$ is continuous and increasing in a neighborhood
of $s$.  We therefore assume that $\{\nu_{n} \}_{n \in \mathbb{N}}$
are $t$-centered (we may do this retroactively since this only
translates our measures $\nu_{n}$ and does not effect the fact that
they cluster to translations of $\rho$).  By \ref{bellemma}, all
cluster points of $\{ \nu_{n} \}_{n \in \mathbb{N}}$ must be
$t$-centered so that, by uniqueness of this property, our sequence
converges to a single measure.

Now, observe that these facts together imply that $\{ \nu_{n,m}
\}_{n,m}$ must converge to the Dirac mass at $0$ as $m \uparrow
\infty$.  This further implies that $\nu_{\infty, m}$ is the limit
of a convergent sequence.  We next claim that if $\mu_{0}$ is any
cluster point of $\{ \mu_{0,n} \}_{n \in \mathbb{N}}$, then $\mu =
\lim_{n \uparrow \infty} \mu_{0} \boxplus \nu_{n}$.

To see this, let $i_{n}$ be a subsequence along which $\mu_{0,n}$
converges to $ \mu_{0}$.  Observe that $\lim_{n \uparrow \infty}
\mu_{0} \boxplus \nu_{n} = \lim_{n \uparrow \infty}  \mu_{0,i_{n}}
\boxplus \nu_{n} = \lim_{n \uparrow \infty} \mu_{0,i_{n}} \boxplus
\nu_{i_{n}} \boxplus \nu_{n , i_{n}} = \lim_{n \uparrow \infty} \mu
\boxplus \nu_{n , i_{n}}$.  As $n \rightarrow \infty$, the right
hand side converges to $\mu \boxplus \delta_{0} = \mu$, proving our
claim.

We have shown that $\mu =\lim_{n \uparrow \infty} \mu_{0} \boxplus
\nu_{n}$ so that our theorem will be proven once we show that
$\mu_{0}$ is infinitely divisible.  Towards this end, we claim that
$\mu_{k,0} = \mu_{0} \boxplus \nu_{\infty , k + 1}$ for all $k \in
\mathbb{N}$. To see this, observe that the right hand side is equal
to $ \mu_{0} \boxplus \nu_{\infty , k + 1} = \lim_{n \uparrow
\infty} \mu_{0,i_{n}} \boxplus \nu_{\infty , k + 1} = \lim_{n
\uparrow \infty} \mu_{0 , i_{n}} \boxplus \nu_{i_{n} , k+1} \boxplus
\nu_{\infty , i_{n}} = \lim_{n \uparrow \infty} \mu_{0,k} \boxplus
\nu_{\infty , i_{n}} \rightarrow \mu_{0,k}$ as $n \rightarrow
\infty$.  This proves our claim.

Now, assume that $\mu_{0}$ has a decomposition $\mu_{0} = \rho
\boxplus \nu$ where $\nu$ is indecomposable.  Assume that $\Lambda
(\nu)
> 0$.  Pick $n$ large enough so that $\alpha_{n} < \Lambda(\nu)$ and
recall that $\mu_{n,0} = \mu_{0} \boxplus \nu_{\infty , n+1}$. The
left hand side has no indecomposable divisior whose $\Lambda$ value
is larger than $\alpha_{n}$.  This contradiction implies that
$\mu_{0}$ has no indecomposable divisors so that, by Theorem
\ref{id}, our theorem holds.
\end{proof}

The failure of uniqueness will be addressed in Section
\ref{applicationsection}.

\section{Background and Terminology for the Multiplicative
Convolution of Measures Supported on the Positive Real
Line}\label{backgroundsubsection}

Let $x,y$ be positive random variables in $(A,\tau)$ with respective
distributions $\mu$ and $\nu$.  We denote by $\mu \boxtimes \nu$ the
distribution of the random variable $xy$. Since $\tau$ is a trace,
the distribution of $xy$ is the same as that of $y^{1/2}xy^{1/2}$,
so that $\boxtimes$ preserves the property that the distribution is
a measure supported on the positive real numbers.

Let $M_{\mathbb{R}^{+}}$ denote the set of probability measures
supported on $\mathbb{R^{+}}$.  Observe that, with exception of
$\delta_{0}$, all such measures have nonzero first moment and we
assume throughout that we are not dealing with this measure.
Consider the following function:
$$ \psi_{\mu}(z) = \int_{0}^{\infty} \frac{zt}{1-zt} d\mu(t) $$ for
$z \in \mathbb{C} \setminus \mathbb{R}^{+}$.  As seen in \cite{V2}
and \cite{BV1}, $\psi_{\mu} |_{i \mathbb{C}^{+}}$ is univalent and
maps into an open neighborhood about the interval $(\mu(\{ 0 \}) - 1
, 0)$. It is also true that $\psi_{\mu}(i \mathbb{C}^{+}) \cap
\mathbb{R} = (\mu(\{ 0 \}) - 1 , 0)$.

Let $\Omega_{\mu} = \psi_{\mu}(i \mathbb{C}^{+})$ and let
$\chi_{\mu}: \Omega_{\mu} \rightarrow i \mathbb{C}^{+}$ denote the
inverse function.  We refer to the $S$-tranform as the following
function:
$$S_{\mu} (z) = \frac{(1+z)\chi_{\mu} (z)}{z}$$  These functions
have the following properties which will be used, often without
reference, in what follows:
\begin{enumerate}
\item
$S_{\mu \boxtimes \nu} (z) = S_{\mu}(z)S_{\nu}(z)$ for all $z$ in
their common domain.

\item
$S_{\mu}(z) > 0$ and $S'_{\mu}(z) \leq 0$ for $z \in (\mu(\{ 0 \}) -
1 , 0)$.

\item
$(\mu_{1} \boxtimes \mu_{2}) (\{ 0 \}) = \max \{\mu_{1}(\{0 \}) ,
\mu_{2}(\{0 \})  \}$

\item
$\chi'_{\mu} (z) > 0$ for all $z \in (\mu(\{ 0 \}) - 1 , 0)$.

\item
$\chi_{\mu \boxtimes \delta_{c}}(z) = \chi_{\mu} (z) / c$ and
$S_{\mu \boxtimes \delta_{c}}(z) = S_{\mu} (z) / c$.

\end{enumerate}

Observe that $(3)$ above implies a multiplicative version of Lemma
\ref{extensionlemma}.  That is, for any nontrivial decomposition
$\mu = \mu_{1} \boxtimes \mu_{2}$, (3) implies that real part of the
domain of $\chi_{\mu}$ is contained in the real part of the domain
of $\chi_{\mu_{i}}$ for each $i = 1,2$.  We will use this fact
without reference throughout.

The following results on convergence and tightness were first proven
in full generality in \cite{BV1}.

\begin{lemma}\label{multcont}
Let $\{ \mu_{n} \}_{n \in \mathbb{N}}$ and $\{ \nu_{n} \}_{n \in
\mathbb{N}}$ be sequence of probability measures on
$\mathbb{R}^{+}$.  Assume the these sequences weakly converge to
$\mu$ and $\nu$ respectively.  Then $\{ \mu_{n} \boxtimes \nu_{n}
\}_{n \in \mathbb{N}}$ converges to $\mu \boxtimes \nu$ in the
weak$^{\ast}$ topology.
\end{lemma}

\begin{lemma}\label{multlemma}
Let $M$ be a set of probability measures on $\mathbb{R}^{+}$.  The
following conditions are equivalent.

\begin{enumerate}

\item
$M$ is tight and the weak$^{\ast}$ closure of $M$ does not contain
$\delta_{0}$.

\item
There exists an $\alpha > 0$ such that
\begin{enumerate}
\item
$-\alpha$ belongs to the domain of $\chi_{\mu}$ for all $\mu \in M$.

\item
$\sup\{ |\chi_{\mu(-\alpha)| : \mu \in M} \} < \infty$

\item
$\inf \{ |\chi_{\mu}(-\beta)| : \mu \in M| \} > 0$ for all $\beta
\in (0 , \alpha)$.

\end{enumerate}

\item
There exists an $\alpha > 0$ such that
\begin{enumerate}
\item
$-\alpha$ belongs to the domain of $S_{\mu}$ for all $\mu \in M$.

\item
$\sup\{ |S_{\mu(-\alpha)| : \mu \in M} \} < \infty$

\item
$\inf \{ |S_{\mu}(-\beta) : \mu \in M| \} > 0$ for all $\beta \in (0
, \alpha)$.

\end{enumerate}

\end{enumerate}

\end{lemma}

\begin{lemma}\label{multlemma2}

Let $\{ \mu_{n} \}_{n \in \mathbb{N}}$ be a tight sequence of
probability measures on $\mathbb{R}^{+}$ such that $\delta_{0}$ is
not in the weak$^{\ast}$ closure of our sequence.  The following are
equivalent:

\begin{enumerate}
\item
The sequence $\{ \mu_{n} \}_{n \in \mathbb{N}}$ converges to a
measure $\mu$ in the weak$^{\ast}$ topology.

\item
There exist positive numbers $\beta < \alpha$ such that the sequence
$\{ \chi_{\mu_{n}} \}$ converges uniformly on the interval $(-\alpha
, -\beta)$ to a function $\chi$.

\item
There exist positive numbers $\beta < \alpha$ such that the sequence
$\{ S_{\mu_{n}} \}$ converges uniformly on the interval $(-\alpha ,
-\beta)$ to a function $S$.

\end{enumerate}

Moreover, if $(1)$ and $(2)$ are satisfied, we have $\chi =
\chi_{\mu}$ in $(-\alpha , -\beta)$.
\end{lemma}

In a manner analagous to the additive case, we have the following
subordination result for multiplicative convolution.  This was first
proven in full generality in \cite{Bi} and is proven by different
means in \cite{BerBel2}.

\begin{theorem}
Let $\mu$ be a probability measure on $\mathbb{R}^{+}$ with
decomposition $\mu = \mu_{1} \boxtimes \mu_{2}$.  There exist
analytic subordination functions $\omega_{i} : \mathbb{C} \setminus
\mathbb{R}^{+} \rightarrow \mathbb{C} \setminus \mathbb{R}^{+}$ for
$i = 1,2$, such that:

\begin{enumerate}
\item
$\omega_{i}(0-) = 0$

\item
for every $\lambda \in \mathbb{C}^{+}$ we have that
$\omega_{i}(\bar{\lambda}) = \overline{\omega_{i}(\lambda)}$ ,
$\omega_{i} (\lambda) \in \mathbb{C}^{+}$ and
$\textrm{arg}(\omega_{j}(\lambda)) \geq \textrm{arg}(\lambda)$

\item
$\psi_{\mu} (\lambda) = \psi_{\mu_{i}} (\omega_{i} (\lambda))$ for
all $\lambda \in \mathbb{C} \setminus \mathbb{R}^{+}$

\item
$\omega_{1}(\lambda) \omega_{2} (\lambda) = \lambda \psi_{\mu}
(\lambda)$

\end{enumerate}

\end{theorem}

Consider next the following result which may be found in
\cite{BerBel}.

\begin{theorem}
Let $\eta: \Omega \rightarrow \mathbb{C} \setminus \{ 0 \}$ be an
analytic function such that $\eta(\bar{z}) = \overline{\eta(z)}$ for
all $z \in \Omega$.  The following are equivalent:

\begin{enumerate}
\item
There exists a probability measure $\mu \neq \delta_{0}$ on $[0 ,
\infty)$ such that $\eta = \psi_{\mu} / (1 + \psi_{\mu})$.

\item
$\eta(0-) = 0$ and $\textrm{arg}(\eta(z)) \in [\textrm{arg}(z) ,
\pi)$ for all $z \in \mathbb{C}^{+}$.

\end{enumerate}

\end{theorem}

These two theorems may be combined to give us the following
corollary.  We have no direct reference for this fact but can be
sure that it is well know and are recording it only for the reader's
convenience.

\begin{corollary}\label{repcorr}
Let $\omega_{i}$ be a subordination function arising from the
decomposition $\mu = \mu_{1} \boxtimes \mu_{2}$ as above.  Then
$$\omega_{i} (z) = \frac{\psi_{\nu} (z)}{ 1 + \psi_{\nu} (z)}$$ for a
probability measure $\nu$ with the property that $\textrm{supp}(\nu)
\subseteq \textrm{supp}(\mu)$

\end{corollary}

\begin{proof}
The existence of such a representation is a direct consequence of
the previous theorems.  It remains to prove to the assertion about
the support of $\nu$.

In the proof of Theorem \ref{multcompacttight} in the next section,
we will show that $\omega_{i}$ will have analytic continuation and
is real on $\mathbb{R} \setminus (\textrm{supp}(\mu)^{-1})$ where
$\textrm{supp}(\mu)^{-1} = \{ t^{-1} : t \in \textrm{supp}(\mu) \}$.
This implies that $\Im \psi_{\nu} (t + i \epsilon) \rightarrow 0$ as
$\epsilon \rightarrow 0$ for $t \notin (\textrm{supp}(\mu)^{-1})$.
Since $G_{\nu} (1/z) = z(\psi_{\nu} (z) + 1)$, this implies that
$t^{-1} \notin \textrm{supp}(\nu)$.  Our claim follows.

\end{proof}

This final result was first proven in \cite{Bel} and will be used in
proving a multiplicative version of Theorem \ref{id}.

\begin{theorem}\label{beltheor}
Consider $(c_{n})_{n \in \mathbb{N}} \subseteq \mathbb{R}$ and an
array $\{ \mu_{n,j} \}_{n \in \mathbb{N} , \ j = 1,2, \ldots k_{n}}$
of probability measures on $(0, \infty)$ such that $$ \lim_{n
\rightarrow \infty} \min_{1 \leq j \leq k_{n}} \mu_{n,j} ((1-
\epsilon , 1 + \epsilon)) = 1
$$ for every $\epsilon > 0$.  If the measures $\delta_{c_{n}} \boxtimes \mu_{n,1} \boxtimes
\cdots \boxtimes \mu_{n,k_{n}}$ have a weak limit $\mu$ which is a
probability measure, then $\mu$ is infinitely divisible.
\end{theorem}

Observe that the assumptions in this theorem may be weakened so that
we need only assume that $\mu_{n,j}(\{0 \}) = 0$ for all $n \in
\mathbb{N}$ and $j = 1,2 , \ldots , k_{n}$.  Indeed, every element
in such an array can be approximated arbitrarily well by a measure
supported on $(0 , \infty)$.  It is under this weakened assumption
that we will later invoke this theorem.

\section{Compactness Results for Measures Supported on the Positive Real
Half-Line}

We define $\textrm{logdiam}(\mu) := \sup_{x , y \in
\textrm{supp}(\mu)} (|\log (x) - \log (y)|)$ to be the
\textit{logarithmic diameter} of the measure $\mu$.
\begin{theorem}\label{multcompacttight}
Let $\mu$ be a compactly supported probability measure on
$\mathbb{R}^{+}$.  Then for any decomposition $\mu = \mu_{1}
\boxtimes \mu_{2}$ we have that $\textrm{logdiam}(\mu_{i}) \leq
\textrm{logdiam}(\mu)$.  If $\mu(\{ 0 \}) = 0$ then equality occurs
if and only if one of the $\mu_{i}$ is a Dirac mass.
\end{theorem}
\begin{proof}
If $\{ 0 \}$ is contained in the support of $\mu$, the theorem is
trivial. Thus, we assume that $[\alpha , \beta] =
\textrm{conv}(\textrm{supp} (\mu))$ and $[\alpha_{1} , \beta_{1}] =
\textrm{conv}(\textrm{supp} (\mu_{1}))$ with $\alpha , \alpha_{1} >
0$. Observe that $\psi_{\mu}$ has analytic extension to $\mathbb{R}
\setminus [\beta^{-1} , \alpha^{-1}]$.  We claim that the
subordination function $\omega_{1}$ does also.

To see this, note that $\psi_{\mu_{1}} (\omega_{1}(te^{i \theta})) =
\psi_{\mu} (t e^{i \theta}) = (G_{\mu} (1 / te^{i \theta}) / t e ^{i
\theta }) + 1$ for $t \in \mathbb{R} \setminus [\beta^{-1},
\alpha^{-1}] $.  Since $1 / t$ is not contained in the support of
$\mu$, the Stieltjes inversion formula tells us that the imaginary
part of the right hand side goes to zero as $\theta$ goes to $0$.
Since $\psi_{\mu_{1}}$ increases argument, the imaginary part of
$\omega_{1}(te^{i \theta})$ must go to zero. The Schwarz reflection
principle implies that $\omega_{1}$ extends analytically across $t$.

As we saw in Corollary \ref{repcorr}, we have that $\omega_{1} (z) =
\psi_{\nu(z)} / (1 + \psi_{\nu(z)})$ for $\nu$ supported on $[\alpha
, \beta]$. Thus, $\omega'_{1}(z) = (\int t(1-zt)^{-2} d \nu (t)) /
(\int (1 - zt)^{-1} d \nu (t))^{2}$  so that $ \lim_{\lambda
\uparrow \infty} \omega'_{1} (\lambda) = (\int t^{-1} d \nu
(t))^{-1}$. We call this limit $\omega'_{1}(\infty)$.

We now claim that $\lambda \omega'_{1}(\infty) - \omega_{1}
(\lambda) \rightarrow C < 0$ as $\lambda \uparrow \infty$.  Indeed,
\begin{eqnarray}
\omega_{1} (\lambda) - \lambda \omega'_{1} (\infty) & = & \frac{
\int_{\alpha}^{\beta} \frac{t\lambda}{1 - t\lambda} d\nu(t) }{
\int_{\alpha}^{\beta} \frac{1}{1 - t\lambda} d\nu(t) } -
\frac{\lambda}{ \int_{\alpha}^{\beta} t^{-1} d\nu(t)} \nonumber \\
& = &  \frac{ \int_{\alpha}^{\beta} t^{-1}d\nu(t)
\int_{\alpha}^{\beta}\frac{t\lambda}{ 1 - t\lambda}d\nu(t) - \lambda
\int_{\alpha}^{\beta} \frac{1}{1 - t\lambda}d\nu(t) }{
\int_{\alpha}^{\beta}t^{-1}d\nu(t) \int_{\alpha}^{\beta}\frac{1}{1 -
t\lambda} d\nu(t) } \nonumber \\
& = &  \lambda \frac{\int_{\alpha}^{\beta}t^{-1}d\nu(t)
\int_{\alpha}^{\beta}\frac{t}{\lambda^{-1} - t}d\nu(t) -
\int_{\alpha}^{\beta}\frac{1}{\lambda^{-1} - t}d\nu(t)
}{\int_{\alpha}^{\beta}t^{-1}d\nu(t)
\int_{\alpha}^{\beta}\frac{1}{\lambda^{-1}- t}d\nu(t)} \nonumber
\\
& = & \lambda \frac{\int_{\alpha}^{\beta}t^{-1}d\nu(t) (1 +
\int_{\alpha}^{\beta}\frac{t}{\lambda^{-1} - t}d\nu(t)) -
(\int_{\alpha}^{\beta}t^{-1}d\nu(t) +
\int_{\alpha}^{\beta}\frac{1}{\lambda^{-1} - t}d\nu(t))}{
\int_{\alpha}^{\beta}t^{-1}d\nu(t)
\int_{\alpha}^{\beta}\frac{1}{\lambda^{-1}-t}d\nu(t) } \nonumber \\
& = & \frac{ \int_{\alpha}^{\beta}t^{-1}d\nu(t)
\int_{\alpha}^{\beta}\frac{1}{\lambda^{-1} - t}d\nu(t) -
\int_{\alpha}^{\beta}\frac{1}{t(\lambda^{-1} -
t)}d\nu(t)}{\int_{\alpha}^{\beta}t^{-1}
\int_{\alpha}^{\beta}\frac{1}{\lambda^{-1} - t}d\nu(t) } \nonumber
\\
& \rightarrow &   \frac{ - (\int_{\alpha}^{\beta}t^{-1}d\nu(t))^{2}
+ \int_{\alpha}^{\beta}t^{-2}d\nu(t)
}{-(\int_{\alpha}^{\beta}t^{-1})^{2}  } = C \nonumber
\end{eqnarray}
as $\lambda \uparrow \infty$. Note that $f(t) = t^{2}$ is a strictly
convex function on $[\alpha , \beta]$. Assuming that $\nu$ is not a
Dirac mass, it follows from Jensen's inequality that C is a strictly
negative number (we may assume that $\nu$ is not a Dirac mass since
this would imply that $\mu_{1}$ is a Dirac mass and our theorem is
trivially true in this case).

Now, by Cauchy-Schwarz, we have that $|\omega'_{1}(z)| \geq
\omega'_{1} (\infty)$ for all $z \in \mathbb{C}^{+} \setminus
[\beta^{-1} , \alpha^{-1}]$. Indeed, we have that $$| \omega'_{1}
(z)| = \left| \frac{\int \frac{t}{(1-zt)^{2}} d \nu (t)}{ (\int
\frac{1}{1 - zt} d \nu (t))^{2}} \right| = \frac{\|
\frac{\sqrt{t}}{z-t} \|_{2}^{2}}{ \left|\left\langle
\frac{1}{\sqrt{t}} , \frac{\sqrt{t}}{z-t} \right\rangle \right|^{2}}
\geq \omega'_{1} (\infty)$$

Thus, $$ \omega_{1} (\alpha^{-1} + \epsilon) = \omega_{1} (\lambda)
- \int_{\alpha^{-1} + \epsilon}^{\lambda} \omega'_{1}(t) dt \ \leq \
\omega_{1} (\lambda) - \lambda \omega'_{1}(\infty) + (\alpha^{-1} +
\epsilon) \omega'_{1}(\infty)$$ which converges to $(\alpha^{-1} +
\epsilon) \omega'_{1}(\infty) + C$ as $\lambda \uparrow \infty$.

To complete our claim, note that $$\omega_{1} (\beta^{-1} -
\epsilon) = \omega_{1} (0) + \int_{0}^{\beta^{-1} - \epsilon}
\omega'(t)dt \leq \omega'_{1}(\infty) (\beta^{-1} - \epsilon)$$
since $\omega_{1} (0) = 0$.  Thus, $  \mathbb{R}^{+} \setminus
[\omega'_{1}(\infty) \beta^{-1} , \omega'_{1}(\infty) \alpha^{-1} +
C] \subseteq \omega_{1} (\mathbb{R}^{+} \setminus [\beta^{-1} ,
\alpha^{-1}])$ . Since $\psi_{\mu_{1}}$ can be continued
analytically to the right hand set, we have that it also has
analytic continuation to $\mathbb{R}^{+} \setminus
[\omega'_{1}(\infty) \beta^{-1} , \omega'_{1}(\infty) \alpha^{-1} +
C]$.  This implies that the support of $\mu_{1}$ is contained in
$([\omega'_{1}(\infty) \beta^{-1} , \omega'_{1}(\infty) \alpha^{-1}
+ C])^{-1} \subseteq \omega'_{1}(\infty)^{-1}[ \alpha , \beta]$ with
equality if and only if one of the $\mu_{i}$ is a Dirac mass. The
theorem follows.
\end{proof}

\begin{theorem}\label{multtight}
Let $\mu$ be a probability measure with $supp (\mu) \subset
\mathbb{R}^{+}$ different from $\delta_{0}$.  Let $\mu = \mu_{1,k}
\boxtimes \mu_{2,k}$ be a family of decompositions.  There exists a
sequence $\{ \lambda_{k} \} \subset \mathbb{R}^{+}$ so that the
families $\{ \mu_{1,k} \circ D_{\lambda_{k}} \}_{k \in \mathbb{N}}$
and $\{ \mu_{2,k} \circ D_{\lambda_{k}^{-1}} \}_{k \in \mathbb{N}}$
are tight.  Furthermore, $\delta_{0}$ is not in the weak closure of
either of these families of measures.
\end{theorem}

\begin{proof}
Let $1 - \mu(\{ 0 \}) = -\alpha < 0$.  Recall that $\psi_{\mu}$ maps
the negative half line injectively onto $(-\alpha , 0)$.  Also
recall that, for each $k$, $\psi_{\mu_{k}}$ maps the negative half
line injectively onto $(1 -\mu_{i,k}(\{0 \}) , 0)$ and that
$\mu_{i,k}(\{ 0 \}) \leq \mu(\{ 0 \})$.  Thus, for each $k$, there
exists a unique real number $\lambda_{k}$ so that $\psi_{\mu_{1,k}
\circ D_{\lambda_{k}}} (-1) = -\alpha / 2$. Denote the new measure
by $\nu_{1,k}$. Dilate $\mu_{2,k}$ by $D_{\lambda_{k}^{-1}}$ and
denote the new measure by $\nu_{2,k}$. Observe that $\mu = \nu_{1,k}
\boxtimes \nu_{2,k}$ for all $k \in \mathbb{N}$.

Now, observe that $-\alpha/2$ is contained in the domain of
$\chi_{\nu_{1,k}}$  and that $|\chi_{\nu_{1,k}} (-\alpha/2)| = 1$
for all $k \in \mathbb{N}$.  By Lemma \ref{multlemma}, if we can
show that $\inf_{k \in \mathbb{N}} |\chi_{\nu_{1,k}} (-\beta)| > 0$
for all $\beta \in (0 , \alpha/2)$, then $\{ \nu_{1,k} \}$ is tight.

Consider the following equation for $t \in (0 , \alpha/2)$:
\begin{equation}\label{eq1}
 \frac{-t + 1}{-t} \chi_{\nu_{1,k}} (-t) \chi_{\nu_{2,k}} (-t) = \chi_{\mu}
 (-t)
\end{equation}
Assume that for $\beta \in (0, \alpha / 2 )$, we have that $\inf_{k
\in \mathbb{N}} (\chi_{\nu_{1,k}}(-\beta)) = 0$.  Our assumption
that $\mu \neq \delta_{0}$ implies that $\chi_{\mu} (-\beta) > 0$.
Manipulating (\ref{eq1}), this implies that $\{
\chi_{\nu_{2,k}}(-\beta) \}$ are unbounded over $k$ and negative. As
$\chi'_{\nu_{2,k}} (t)
> 0$, this implies that $\{ \chi_{\nu_{2,k}} (-\alpha / 2) \}$ are
unbounded over $k$. However, (\ref{eq1}) and the assumption that
$\chi_{\nu_{1,k}}(-\alpha / 2) \equiv -1$ results in contradiction.
By Lemma \ref{multlemma}, $\{ \nu_{1,k} \}$ is a tight family.

It is easily seen that $\{ \nu_{2,k} \}$ is also a tight family.
Indeed, $\chi_{\nu_{1,k}}(-\alpha / 2) \equiv -1$ implies that
$$\chi_{\nu_{2,k}}(-\alpha / 2) \equiv \frac{\alpha / 2}{1 - \alpha/2} \chi_{\mu} (- \alpha /2)$$
Thus, the first two criteria of Lemma \ref{multlemma} are satisfied
and the last follows from the fact that for fixed $\beta \in (0 ,
\alpha / 2)$,
$$|\chi_{\nu_{2,k}}(-\beta)| = \frac{\beta \chi_{\mu}
 (-\beta)}{ (1 - \beta)\chi_{\nu_{1,k}}(-\beta) } \geq \frac{-\beta \chi_{\mu}
 (-\beta)}{ (1 - \beta) }  > 0$$
\end{proof}

\section{A Khintchine Decompostion for Multiplicative Free Convolution with Measures Supported on the Postive Half Line}

\begin{theorem}\label{multid}
Let $\mu$ be a probability measure with the property that, for any
non trivial decomposition $\mu = \mu_{1} \boxtimes \mu_{2}$, neither
$\mu_{1}$ nor $\mu_{2}$ is indecomposable.  Then $\mu$ is
$\boxtimes$-infinitely divisible.
\end{theorem}
\begin{proof}

Let $\alpha = 1 - \mu(\{ 0 \} )$. We will show later that $\alpha =
1$. Recall that $S_{\mu}$, $S_{\mu_{1}}$ and $S_{\mu_{2}}$ are all
defined on an open neighborhood of $(-\alpha , 0)$ for any
decomposition $\mu = \mu_{1} \boxtimes \mu_{2}$. We assume without
loss of generality that $S_{\mu} (-\beta) = 1$ for some $\beta \in
(0, \alpha)$ (indeed, pick any $\beta$ in this interval and then
consider $\mu \boxtimes \delta_{c}$ where $c = (-\beta
\chi_{\mu}(-\beta)) / (1 - \beta) \ $).

We denote by $M_{\beta}$ the set of all probability measures $\nu
\in M_{\mathbb{R}^{+}}$ such that $S_{\nu}(-\beta) = 1$ and $\mu =
\nu \boxtimes \rho$ for a probability measure $\rho \in
M_{\mathbb{R}^{+}}$.  Observe that $S_{\mu}(-\beta) =
S_{\nu}(-\beta) S_{\rho}(-\beta)$ implies that $\rho \in M_{\beta}$.
Further note that for any decomposition $\mu = \nu' \boxtimes \rho'$
there exists a real number $c$ such that $\nu' \boxtimes \delta_{c}$
, $\rho \boxtimes \delta_{c^{-1}} \in M_{\beta}$.  Lastly, it is the
content of Theorem \ref{multtight} that $M_{\beta}$ is weak$^{\ast}$
compact.

Fix $\gamma \in (0 , \beta)$.  We claim that given any $\epsilon
> 0$, there exists an element $\nu \in M_{\beta}$ such that
$1 > S_{\nu} (-\gamma) > 1 - \epsilon$.  To show this, assume
instead that there is a $\delta
> 0$ so that $1 - \delta$ is the supremum of $S_{\nu} (-\gamma)$ ranging over all nontrivial elements
 in $M_{\beta}$. By compactness, we may pass to a
cluster point, and assume that we have a decomposition $\mu =
\mu_{1} \boxtimes \mu_{2}$ where $S_{\mu_{1}}(-\gamma)$ takes on
this supremum. Now, by assumption, we have a nontrivial
decomposition $\mu_{1} = \nu_{0} \boxtimes \nu_{1}$ where
$S_{\nu_{i}} (-\beta) = 1$ for $i = 0,1$. Since $S'_{\nu_{i}} \leq
0$, this implies that both $S_{\nu_{i}} (-\gamma) < 1$ (we would
have equality if and only if $\nu_{i}$ were a Dirac mass, which we
have assumed away). As their product satisfies $S_{\nu_{0}}
(-\gamma) S_{\nu_{1}} (-\gamma) = S_{\mu_{1}} (-\gamma) = 1 -
\delta$, we have that $S_{\nu_{i}} (-\gamma) > 1 - \delta$ for $i =
1,2$. Thus, the decomposition $\nu_{0} \boxtimes (\nu_{1} \boxtimes
\mu_{2})$ violates the above supremum.

We next claim that $S_{\nu} (-\gamma)$ takes on all values of the
interval $[S_{\mu} (-\gamma) , 1]$ as we range over elements in
$M_{\beta}$.  Clearly our compactness result implies that the range
of the $S_{\nu}(-\gamma)$ is closed.  We assume, for the sake of
contradiction, that there exists real numbers $\delta > 0$ and
$\lambda
> S_{\mu} (-\gamma)$ such that $S_{\nu}(-\gamma)$ does not take on
any values in the interval $(\lambda - \delta , \lambda)$ for $\nu
\in M_{\beta}$ and that this interval is maximal in this regard.
Passing to cluster points, we assume that $S_{\mu_{1}} (-\gamma) =
\lambda$ for a decomposition $\mu = \mu_{1} \boxtimes \mu_{2}$. Now,
pick a nontrivial decomposition $\mu_{2} = \nu_{0} \boxtimes
\nu_{1}$ so that $S_{\nu_{0}} (-\gamma)$ is close enough to $1$  so
that $\lambda S_{\nu_{0}} (-\gamma) \in (\lambda - \delta ,
\lambda)$. Transferring this mass, we obtain our contradiction.

By induction,  there exists a decomposition $\mu = \mu_{n,1}
\boxtimes \cdots \boxtimes \mu_{n,n}$ such that $S_{\mu_{n,i}}
(-\beta) = 1$ and $S_{\mu_{n,i}} (-\gamma) =
\sqrt[n]{S_{\mu}(-\gamma)} $ for all $n \in \mathbb{N}$ and $i =
1,2, \ldots, n$. Observe that this implies that $S_{\mu_{n,i}}(-t)
\rightarrow 1$ uniformly for $t \in (\gamma , \beta)$ and $n \in
\mathbb{N}$ ($S_{\mu_{n,i}}$ is non-increasing on this interval). By
Lemma \ref{multlemma2} this implies that any subsequence of our
array $\{ \mu_{n,i} \}_{n \in \mathbb{N} , \ i = 1,2,\ldots, n}$
converges to $\delta_{1}$. Compactness implies that our array
converges to $\delta_{1}$ uniformly over $n$.  Lastly, note that
this implies that our measures satisfy $\mu_{n,i}(\{ 0\}) = 0$.
Indeed, observe that $\max_{i = 1,2, \ldots , n} \mu_{n,i} (\{ 0 \})
\rightarrow 0$ since we have uniform weak convergence to
$\delta_{1}$.  Since $\mu (\{ 0 \}) = \max_{i = 1,2, \ldots , n}
\mu_{n,i} (\{ 0 \})$ we must have no mass at $0$ for $\mu$ or for
any element in our array.

Thus, we may now invoke Theorem \ref{beltheor} which implies that
our measure $\mu$ is $\boxtimes$-infinitely divisible.
\end{proof}

\begin{theorem}\label{multdecomp}
Let $\mu \in M_{\mathbb{R_{+}}}$ different from $\delta_{0}$.  Then
there exist measures $\mu_{i}$ with $i = 0,1,2, \ldots $ such that
$\mu_{0}$ is $\boxtimes$-infinitely divisible, $\mu_{i}$ is
$\boxtimes$-indecomposable for $i = 1, 2, \ldots $, and $\mu =
\mu_{0} \boxtimes \mu_{1} \boxtimes \mu_{2} \boxtimes \cdots$. This
decomposition is not unique.
\end{theorem}
\begin{proof}
We again assume without loss of generality that $S_{\mu} (-\beta) =
1$ for some $\beta \in (0 , 1 - \mu(\{ 0 \}))$.  In what follows,
all decompositons will be taken from elements in $M_{\beta}$

Pick $\gamma \in (0,\beta)$.  Let $\alpha = S_{\mu} (-\gamma) \leq
1$ (with equality if and only if $\mu = \delta_{1}$, in which case
the theorem is trivially true).  Now, let $\alpha_{0} = \inf \{
S_{\nu} (-\gamma) \}$ where the infimum is taken over all
indecomposable $\nu \in M_{\beta}$.  If $\alpha_{0} = 1$ then, by
Theorem \ref{multid}, our theorem holds.  If not, let $\mu =
\mu_{0,1} \boxtimes \mu_{1}$ with $\mu_{1} \in M_{\beta}$
indecomposable satisfying $S_{\mu_{1}} (-\gamma) >
\sqrt{\alpha_{0}}$.

At the $n$th stage of this process, we start with a decompostion
$\mu = \mu_{0, n-1 } \boxtimes \mu_{1} \boxtimes \mu_{n-1}$ where
all divisors are elements of $M_{\beta}$ and $\mu_{i}$ is
indecomposable for $i = 1,2, \ldots, n-1$.  We let $\alpha_{n-1} =
\inf \{ S_{\nu} (-\gamma) \}$ where the infimum is taken over all
indecomposable $\nu \in M_{\beta}$ such that $\mu_{0, n-1} = \nu
\boxtimes \rho$ for some $\rho \in M_{\beta}$ (observe that $\mu_{0,
n-1} , \ \nu \in M_{\beta}$ implies that $\rho \in M_{\beta}$) . If
at any point $\alpha_{n} = 1$ then, by Theorem \ref{multid}, we are
done. Thus, we assume that $\alpha_{n} < 1$ for all $n \in
\mathbb{N}$. Let $\mu_{0,n-1} = \mu_{0,n} \boxtimes \mu_{n}$ where
$\mu_{n} \in M_{\beta}$ is indecomposable and satisfies $S_{\mu_{n}}
(-\gamma) > \sqrt{\alpha_{n}}$.  At this point, we have a
decomposition $\mu = \mu_{0,n} \boxtimes \mu_{1} \boxtimes \cdots
\boxtimes \mu_{n}$ satisfying $\mu_{0,n}$ , \ $\mu_{i} \in
M_{\beta}$, $\mu_{i}$ is indecomposable and $S_{\mu_{i}}(-\gamma) >
\sqrt{\alpha_{i}}$ for all $i = 1,2, \ldots , n -1 $.

In what follows, we will use the following notation for $n > m$:
$$\nu_{n} = \mu_{1} \boxtimes \cdots \boxtimes \mu_{n}$$
$$\nu_{n,m} = \mu_{m+1} \boxtimes \cdots \boxtimes \mu_{n} $$
$$\nu_{\infty , m} = \lim_{n \uparrow \infty} \mu_{m+1} \boxtimes \cdots \boxtimes \mu_{n}$$
We will show later that this last element actually converges to a
measure in $M_{\beta}$.

Now, observe that $\{ \nu_{n,m} \}_{m < n \in \mathbb{N}}$ is a
tight family since it is a subset of $M_{\beta}$.  We claim that
$\nu_{n,m} \rightarrow \delta_{1}$ uniformly in the weak$^{\ast}$
topology as $m \uparrow \infty$.  Indeed, observe that
$S_{\mu_{0,n}}(-\gamma)$ is increasing and bounded by $1$ which
implies convergence.  Furthermore,
$$S_{\mu} (-\gamma) = S_{\mu_{0,n}}(-\gamma) \ast  S_{\nu_{n}}(-\gamma) =
S_{\mu_{0,n}}(-\gamma) \ast
 S_{\nu_{m}}(-\gamma) \ast S_{\nu_{n,m}}(-\gamma)$$so that
$S_{\nu_{n,m}}(-\gamma)$ represents the tail of a convergent
product.  This implies that $S_{\nu_{n,m}}(-\gamma) \rightarrow 1$
uniformly over $n\in \mathbb{N}$ as $m \uparrow \infty$ (observe
that this also implies that $\alpha_{n} \uparrow 1$). By Lemma
\ref{multlemma2}, any convergent subsequence must converge to
$\delta_{1}$.  By tightness, we must have uniform convergence to
$\delta_{1}$ as $m \uparrow \infty$.

Now, let $\mu_{0}$ be a cluster point of $\mu_{0,n}$.  We claim that
$\mu_{0} \boxtimes \nu_{n} \rightarrow \mu$ in the weak$^{\ast}$
topology.  Indeed, let $i_{k}$ be a subsequence on which $\mu_{0,
i_{k}}$ converges to $\mu_{0}$ and let $f$ map $\mathbb{N}$ onto
this subsequence by letting $f(n) = i_{k}$ where $i_{k} \leq n <
i_{k+1}$.  We then have that $\lim_{n \uparrow \infty} \mu_{0}
\boxtimes \nu_{n} = \lim_{n \uparrow \infty} \mu_{0,f(n)} \boxtimes
\nu_{n} = \lim_{n \uparrow \infty} \mu_{0,f(n)} \boxtimes \nu_{f(n)}
\boxtimes \nu_{n, f(n)} = \lim_{n \uparrow \infty} \mu \boxtimes
\nu_{n, f(n)}$.  As we saw in the previous paragraph, the right hand
side converges to $\mu$.

It remains to show that $\mu_{0}$ is infinitely divisible.  As in
Theorem \ref{k1}, we will show that $\mu_{0} \boxtimes
\nu_{\infty,n} = \mu_{n,0}$.  Indeed, note that $\mu_{0} \boxtimes
\nu_{\infty,n} = \lim_{k \uparrow \infty} \mu_{0, i_{k}} \boxtimes
\nu_{\infty , n} = \lim_{k \uparrow \infty} \mu_{0, i_{k}} \boxtimes
\nu_{i_{k} , n} \boxtimes \nu_{\infty , i_{k}} = \lim_{k \uparrow
\infty} \mu_{n,0} \boxtimes \nu_{\infty , i_{k}} = \mu_{n,0}$,
proving our claim (the second to last equality follows from the fact
that, by construction, $\mu_{0,n} \boxtimes \nu_{n,m} = \mu_{0,m}$
for all $m < n \in \mathbb{N}$).

To complete the proof, assume that $\mu_{0} = \nu \boxtimes \rho$
where $\nu$ is indecomposable and satisfies $S_{\nu}(-\gamma) < 1$.
Pick $n$ such that $\alpha_{n} < S_{\nu}(-\gamma)$.  As $\mu_{0,n} =
\mu_{0} \boxtimes \nu_{\infty , n}$ and the left hand side has no
indecomposable divisors satisfying the above inequality, we have a
contradiction.  Thus, $\mu_{0}$ has no nontrivial divisors so that,
by Theorem $\ref{multid}$, our theorem holds.
\end{proof}

\section{Background and Terminology for Measures Supported on the Unit Circle}

 Let $M_{\mathbb{T}}$ be the
set of all Borel probability measures supported on the unit circle.
Let $M_{\ast}$ be the set of all Borel probability measures on
$\mathbb{C}$ with nonzero first moment.  For a measure $\mu \in
M_{\ast} \cap M_{\mathbb{T}}$ the following definition:
$$\psi_{\mu}(z) = \int_{\mathbb{T}} \frac{zt}{1-zt}d\mu(t) $$
Observe that $\psi_{\mu} (0)= 0$ and $\psi'_{\mu} (0) = \int_{C} t
d\mu_{t}$ so that our assumption of nonzero first moment implies
that $\psi_{\mu}^{-1} = \chi_{\mu}$ is defined and analytic in
neighborhood of $0$.  We again define $S_{\mu}(z) =
(1+z)\chi_{\mu(z)} / z$.  Observe that $S_{\mu} (0) = 1 /
\psi'_{\mu}(0)$ so that $S_{\mu}$ is also defined and analytic in a
neighborhood of $0$. Further note that
$$|\psi'_{\mu} (0)| = \left|\int_{\mathbb{T}} \zeta d\mu (\zeta)\right| \leq \int_{\mathbb{T}} |\zeta| d\mu (\zeta) =1 $$
which implies that $|S_{\mu} (0)| \geq 1$ for $\mu \in M_{\ast} \cap
M_{\mathbb{T}}$.

We now record the following lemmas and theorems for use in proving
our main results. These were first proven in \cite{V2}, \cite{BV2}
and \cite{Bel}.

\begin{lemma}\label{circmult0}
Let $\mu \in M_{\ast} \cap M_{\mathbb{T}}$ satisfy $|S_{\mu} (\{ 0
\})| = 1$.  Then $\mu = \delta_{\alpha}$ for some $\alpha \in
\mathbb{T}$

\end{lemma}

\begin{lemma}\label{circconv}
Let $\mu_{i} \in M_{\ast} \cap M_{\mathbb{T}}$ be such that
$S_{\mu_{i}}(z)$ converge uniformly in some neighborhood of $0$ to a
function $S(z)$.  Then there exists $\mu \in M_{\ast} \cap
M_{\mathbb{T}}$ such that $S = S_{\mu}$

\end{lemma}

\begin{theorem}\label{circconv2}
Consider $\mu \in M_{\ast} \cap M_{\mathbb{T}}$ and let $\mu_{i} \in
M_{\mathbb{T}}$ for $i \in \mathbb{N}$.  If $\mu_{i}$ converge to
$\mu$ in the weak$^{\ast}$ topology, the $\mu_{i} \in M_{\ast} \cap
M_{\mathbb{T}}$ eventually and the functions $S_{\mu_{i}}$ converge
to $S_{\mu}$ uniformly in some neighborhood of zero.  Conversely, if
$\mu_{i} \in M_{\ast} \cap M_{\mathbb{T}}$ and $S_{\mu_{i}}$
converge to $S_{\mu}$ uniformly in some neighborhood of zero then
the measures $\mu_{i}$ converge to $\mu$ in the weak$^{\ast}$
topology.
\end{theorem}

\begin{theorem}\label{circarray}
Let $c_{n} \in \mathbb{T}$ be a sequence of numbers and $\{
\mu_{n,j} \}_{n \in \mathbb{N} , \ j = 1,\ldots, k_{n}}$ be and
array of probability measures in $M_{\mathbb{T}}$ such that $\lim_{n
\uparrow \infty} \max_{j = 1, \ldots , k_{n}} \mu_{n,j} (\{z : \
|z-1| < \epsilon \}) = 1$ for every $\epsilon > 0$.  If the measures
$\delta_{c_{n}} \boxtimes \mu_{n,1} \boxtimes \cdots \boxtimes
\mu_{n, k_{n}}$ have a weak limit $\mu$, then $\mu$ is
$\boxtimes$-infinitely divisible.
\end{theorem}

\section{Main Results for Measures Supported on the Unit Circle}

The last case considered are measures $\mu \in M_{\mathbb{T}} \cap
M_{\ast}$ where $M_{\mathbb{T}}$ are those probability measures
supported on the complex circle and $M_{\ast}$ are those probability
measures with non-zero first moment.  Observe that our
decompositions will be supported on the unit circle so that a family
of decompositions $\mu = \mu_{1,k} \boxtimes \mu_{2,k}$ are
trivially tight.

\begin{theorem}
Let $\mu \in M_{\mathbb{T}} \cap M_{\ast}$ have the property that,
for any non-trivial decomposition $\mu = \nu \boxtimes \omega$ with
$\nu , \omega \in M_{\mathbb{T}} \cap M_{\ast} $, neither $\nu$ nor
$\omega$ is indecomposable. Then $\mu$ is $\boxtimes$-infinitely
divisible.
\end{theorem}
\begin{proof}
Let $\Lambda : M_{\mathbb{T}} \rightarrow \mathbb{C}$ be defined by
$\Lambda (\nu) = S_{\nu}(0)$.  Observe that $| \Lambda (\mu) | \geq
1$ with equality if and only if $\mu$ is a Dirac mass situated on
the circle. We may then assume that $|\Lambda (\mu)| = 1 + \alpha >
1$. In a manner analogous to Theorems \ref{id} and \ref{multid}, for
every $ \alpha
> \epsilon
> 0$, there exists a nontrivial decomposition $\mu = \nu \boxtimes \omega$ such
that $|\Lambda(\nu)| < 1 + \epsilon$.  Through a similar maximality
argument, one can show that for every $n \in \mathbb{N}$ there
exists a decomposition $\mu = \mu_{n,1} \boxtimes \cdots \boxtimes
\mu_{n,n}$ such that $|\Lambda (\mu_{n,i})| =
\sqrt[n]{|\Lambda(\mu)|}$ for all $i = 1, 2, \ldots, n$.  We forgo
the proof due to extreme similarity to the first two cases.

Now, observe that $\Lambda (\mu_{n,i} \boxtimes \delta_{c}) =
\Lambda(\mu_{n,i}) / c$ for $c \in \mathbb{T}$.  Thus, we may assume
that $\mu = \delta_{c_{n}} \boxtimes \mu_{n,1} \boxtimes \cdots
\boxtimes \mu_{n,n}$ for all $n \in \mathbb{N}$ where we
additionally assume that $\Lambda (\mu_{n,i}) =
\sqrt[n]{|\Lambda(\mu)|}$.

Note that $\{ \mu_{n,j} \}_{n \in \mathbb{N} , j = 1,2, \ldots, n}$
forms a tight array since all of our measures are compactly
supported. Further observe that, by Theorem \ref{circconv2} any
cluster point $\nu$ of this array satisfies $\Lambda (\nu) = 1$. By
Lemma \ref{circmult0}, this implies that $\nu = \delta_{1}$.
Tightness implies that our array converges to $\delta_{1}$ uniformly
over $n$. By Theorem \ref{circarray}, this implies
$\boxtimes$-infinite divisibility.
\end{proof}

We close with our Khinthine decomposition for measures in
$M_{\mathbb{T}}$.  Several steps of the proof are indistinguishable
from Theorem \ref{multdecomp} so are not presented in full detail.

\begin{theorem}
Let $\mu \in M_{\mathbb{T}} \cap M_{\ast}$ be a probability measure.
There exists a decomposition $\mu = \mu_{0} \boxtimes \mu_{1}
\boxtimes \mu_{2} \boxtimes \cdots$ such that $\mu_{i} \in
M_{\mathbb{T}} \cap M_{\ast}$ for all $i = 0, 1, 2, \ldots$,
$\mu_{0}$ is infinitely divisible and $\mu_{i}$ is indecomposable
for $i = 1,2, \ldots$.  Such a decomposition need not be unique.
\end{theorem}
\begin{proof}
In a manner entirely analogous with the previous cases, for all $n
\in \mathbb{N}$, we construct a decomposition
$$ \mu = \mu_{0,n} \boxtimes \mu_{1} \boxtimes \cdots \boxtimes
\mu_{n}$$  with the following properties:
\begin{enumerate}
\item
The measure $\mu_{i} \in M_{\mathbb{T}}$ is indecomposable for all
$i \in \mathbb{N}$.
\item
Let $\alpha_{i-1} = \sup |\Lambda(\nu)|$ where the supremum is taken
over all indecomposable measures $\nu \in M_{\mathbb{T}}$ satisfying
$\mu_{0,1} = \nu \boxtimes \rho$ for some $\rho \in M_{\mathbb{T}}$.
We have that $1 \leq \Lambda(\mu_{i}) < \sqrt{\alpha_{i}}$ (in
particular, we may assume that $\Lambda(\mu_{i})$ is real).
\end{enumerate}

We again define $\nu_{n}$ , $\nu_{n,m}$ and $\nu_{\infty , m}$ as in
the proof of Theorem \ref{multdecomp}.  That is
$$\nu_{n} = \mu_{1} \boxtimes \cdots \boxtimes \mu_{n}$$
$$\nu_{n,m} = \mu_{m+1} \boxtimes \cdots \boxtimes \mu_{n} $$
$$\nu_{\infty , m} = \lim_{n \uparrow \infty} \mu_{m+1} \boxtimes \cdots \boxtimes \mu_{n}$$
Observe that tightness is trivial in this case since
$M_{\mathbb{T}}$ is compact. We then have that $\Lambda(\mu) =
\Lambda(\mu_{0,n}) \ast \Lambda(\nu_{n}) = \Lambda(\mu_{0,n}) \ast
\Lambda(\nu_{m}) \ast \Lambda(\nu_{n,m})$. Since
$\Lambda_{\mu_{0,n}}$ is decreasing and bounded as $n \uparrow
\infty$, this is a convergent sequence.  This implies that
$\nu_{n,m}$ represents the tail of a convergent product so that it
goes to $0$ as $m \uparrow \infty$ (this implies that $\alpha_{i}
\rightarrow 1$).  Thus, $\{ \nu_{n,m} \}_{m < n \in \mathbb{N}}$ is
tight and any cluster point $\nu$ of a subsequence with unbounded
$m$ must satisfy $\Lambda(\nu) = 1$.  By \ref{circmult0}, $\nu =
\delta_{1}$.  This implies that $\nu_{n,m} \rightarrow \delta_{1}$
uniformly as $m \uparrow \infty$.

Once again, we let $\mu_{0}$ be a cluster point of $\{ \mu_{0,n}
\}_{n \in \mathbb{N}}$.  In the same manner as in Theorem
\ref{multdecomp}, we have that $\mu_{0} \boxtimes \nu_{n}
\rightarrow \mu$ as $n \uparrow \infty$.

The theorem is proved when we can show that $\mu_{0}$ is infinitely
divisible.  It is again true that $\mu_{0} \boxtimes \nu_{\infty:m}
= \mu_{0, m-1}$ with no deviation from the previous proof. Our
result then follows by the same line of reasoning as Theorem
$\ref{multdecomp}$.
\end{proof}

\section{Applications}\label{applicationsection}

We begin by extending the class of $\boxplus$-indecomposable
measures.

\begin{theorem}
Let $\mu$ be a measure with the property that the left and right
endpoints of the support of $\mu$ are Dirac masses.  Then $\mu$ is
indecomposable.
\end{theorem}
\begin{proof}
Assume that $\mu = \mu_{1} \boxplus \mu_{2}$ and that the support of
$\mu$ has respective left and right endpoints $a$ and $b$.  Recall
that Theorem \ref{belmaintheorem} states that $$\mu(\{ a \}) =
\mu_{1}(\{ a_{1} \} ) + \mu_{2}(\{ a_{2} \}) - 1$$ $$\mu(\{ b \}) =
\mu_{1}(\{ b_{1} \} ) + \mu_{2}(\{ b_{2} \}) - 1$$ for masses $a_{i}
, b_{i} \in supp(\mu_{i})$, and that these points satisfy $a = a_{1}
+ a_{2}$ and $b = b_{1} + b_{2}$.  Now, if $a_{1} \neq b_{1}$ then
$\mu_{1} (\{ a_{1} \}) + \mu_{1} (\{ b_{1} \}) \leq 1$.  Thus, $$ 0
< \mu(\{ a \}) + \mu(\{ b \}) = \mu_{1}(\{ a_{1} \} ) + \mu_{2}(\{
a_{2} \}) + \mu_{1}(\{ b_{1} \} ) + \mu_{2}(\{ b_{2} \})- 2 $$ $$
\leq \mu_{2} (a_{2}) + \mu_{2}(\{ b_{2} \}) - 1$$ Thus, $\mu_{2}
(a_{2}) + \mu_{2}(\{ b_{2} \}) > 1$ so that $a_{2} = b_{2}$.
Translating our measures, we may assume that $a_{2} = b_{2} = 0$.
Thus, $a_{1} = a$ and $b_{1} = b$. This implies that
$\textrm{diam}(supp(\mu_{1})) \geq \textrm{diam}(supp(\mu))$. By
Theorem \ref{tightcompact}, it follows that $\mu_{2} = \delta_{0}$
so that $\mu$ is indecomposable.
\end{proof}

Now, given a measure $\mu$, it was proven by Nica and Speicher in
\cite{NS} that we may associate to $\mu$ a semigroup of measures $\{
\mu_{t} \}_{t \geq 1}$ so that $\mu_{1} = \mu$ and $\mu_{s + t} =
\mu_{s} \boxplus \mu_{t}$.  In particular, $\mu_{n} = \mu \boxplus
\cdots \boxplus \mu$, the n-fold free convolution.  When $\mu$ is
infinitely divisible, this family may be extended to $t \in
\mathbb{R}^{+}$.

It was shown in \cite{BelBer3} that for $\mu = (\delta_{1} +
\delta_{-1}) / 2$, we have that $\mu_{t}$ is a sum of two atoms
concentrated at $\pm t$ and an absolutely continuous measure
concentrated on $[-2\sqrt{t-1} , 2 \sqrt{t-1}]$.  This implies the
following corollary to our theorem.

\begin{corollary}
For $\mu = (\delta_{1} + \delta_{-1}) / 2$, the elements of the
family of measures $\{ \mu_{t} \}_{t \in [1 , 2)}$ are
indecomposable.
\end{corollary}

Observe that this family of examples also dashes any hope of
uniqueness for our Khintchine decomposition.  Indeed, for $\mu$ and
$\{ \mu_{t} \}_{t \geq 1}$ as in the previous example we have that,
for $s = 2 + \epsilon$,  $\mu_{s} = \mu_{t} \boxplus \mu_{s - t}$
for all $t \in (1,1 + \epsilon)$.  This is an uncountable family of
distinct decompositions of $\mu_{s}$ into a sum of indecomposable
elements.

Note that the even the infinitely divisible divisor in the
Khintchine composition cannot be determined uniquely.  Indeed denote
by $\mu$ the semicircle distribution with mean $0$ and variance $1$,
an infinitely divisible measure. It was shown in \cite{BV3} the
there is a nontrivial decomposition $\mu = \nu \boxplus \rho$ where
neither $\nu$ nor $\rho$ is infinitely divisible.  Taking the
Khintchine decompositions for each $\nu$ and $\rho$ and combining
the respective infinitely divisible divisors, we obtain a
decomposition $\mu = \mu_{0} \boxplus \mu_{1} \boxplus \mu_{2}
\boxplus \cdots$ such that $\mu_{0}$ infinitely divisible, $\mu_{i}$
indecomposable for $i \geq 1$ and $\mu_{1}$ nontrivial.  This
implies that $\mu \neq \mu_{0}$.

Lastly, it has come to the author's attention that these results
have been addressed independently in \cite{CG}.  They rightly point
out the following improvement on Theorems \ref{id} and \ref{k1}.
Namely, the class of measures that satisfy the hypotheses of Theorem
\ref{id} are precisely the Dirac measures.  For a simple
justification of this fact, note that we have shown that such
measures are necessarily infinitely divisible. It was shown in
\cite{BV1} that infinitely divisible measures may be decomposed into
the free convolution of a semicircular measure and a free Poisson
measure. Free Poisson measures have indecomposable divisors, almost
by definition.  As was shown in \cite{BV3}, semicircular measures
also have indecomposable divisors.  These facts taken together imply
the above statement so that Theorem \ref{k1} may be improved into a
purely prime decomposition, with no infinitely divisible component.

\section*{Acknowledgements}
I would like to thank my advisor, Hari Bercovici, for his help, his
patience and his numerous suggestions.  I would also like to thank
the referee for his thoughtful recommendations.

\bibliographystyle{abbrv}
\bibliography{khintchine_final_ims_format}

\end{document}